\documentclass[11pt,a4paper]{article}

\usepackage{lipsum}
\usepackage{graphicx}
\usepackage{epstopdf}
\usepackage{algorithm}
\usepackage{algorithmic}
\ifpdf
  \DeclareGraphicsExtensions{.eps,.pdf,.png,.jpg}
\else
  \DeclareGraphicsExtensions{.eps}
\fi

\usepackage{exscale,ifthen,array,amsfonts,amsmath,amscd,amssymb,amsthm}

\usepackage{enumerate}

\usepackage[english]{babel}

\usepackage{color}

\usepackage[shadow]{todonotes} 

\usepackage{standalone} 

\usepackage{hyperref}

\allowdisplaybreaks[2]

\DeclareMathOperator{\cl}{cl}
\DeclareMathOperator{\Id}{Id}

\DeclareMathOperator{\supp}{supp}

\DeclareMathOperator{\Prb}{\mathbf{P}}
\DeclareMathOperator{\Mean}{\mathbf{E}}

\newcounter{hypcount}
\newenvironment{hypenv}{ \begin{enumerate}\setcounter{enumi}{\value{hypcount}}}
{\setcounter{hypcount}{\value{enumi}}\end{enumerate}}

\newcounter{hypcount2}

\newcounter{hypcount3}

\newcommand{\hypref}[1]{(A\ref{#1})}
\newcommand{\hyprefall}{(A1)\,--\,(A\arabic{hypcount})}

\newcounter{dummy}
\newcommand{\hyprefallbutlast}{\setcounter{dummy}{\value{hypcount}}
\addtocounter{dummy}{-1}(A1)\,--\,(A\arabic{dummy})}

\theoremstyle{plain}
\newtheorem{theorem}{Theorem}
\newtheorem{proposition}{Proposition}
\newtheorem{lemma}{Lemma}

\theoremstyle{definition}

\newtheorem*{definition}{Definition}

\theoremstyle{remark}

\newtheorem*{exmpl}{Example}
\newtheorem{example}{Example}
\newtheorem{rem}{Remark}

\renewcommand{\phi}{\varphi}

\newcommand{\eqdef}{\doteq}

\newcommand{\new}[1]{#1}

\makeatletter
\newcommand{\easyhyp}[1]{%
\item[#1]\protected@edef\@currentlabel{#1}%
}
\makeatother

\title{An algorithm to construct subsolutions \\of convex optimal control problems}

\author{
  Gianmarco Bet
  \thanks{
    Department of Mathematics and Computer Science ``Ulisse Dini", University of Florence, Viale Giovan Battista Morgagni 67, 50134 Firenze, Italy
    (\href{mailto:gianmarco.bet@unifi.it}{gianmarco.bet@unifi.it}, \url{http://people.dimai.unifi.it/bet/}).
    }
  \and 
  Markus Fischer
    \thanks{Department of Mathematics, University of Padua, via Trieste 63, 35121 Padova, Italy 
    (\href{mailto:fischer@math.unipd.it}{fischer@math.unipd.it}, \url{  https://www.math.unipd.it/~fischer/}).
    }
}

\usepackage{amsopn}

\begin{document}
\maketitle

\begin{abstract}
  We propose an algorithm that produces a non-decreasing 
sequence of subsolutions for a class of optimal control problems distinguished by the property that the associated Bellman operators preserve convexity. In addition to a theoretical discussion and proofs of convergence, numerical experiments are presented to illustrate the feasibility of the method.
\end{abstract}

\medskip {\small \textbf{MSC2020 subject classifications:}   (Primary) 90C39, 93E20, (Secondary) 49L20, 49M25.}



\section{Introduction}\label{SectIntro}

In this paper, we introduce an iterative scheme that globally approximates the value functions associated with \new{certain} discrete time deterministic or stochastic optimal control problems. At each step our algorithm constructs a subsolution of the control problem as the maximum of a collection of elementary subsolutions. Significantly, this procedure does not rely on a \new{state-space} discretization scheme. \new{As a consequence, when the dimension of the noise space is not too large, our algorithm can efficiently handle high-dimensional problems.} On one hand, the subsolution produced by the algorithm constitutes a lower bound of the value function \new{on the entire state space}. On the other hand, computing the costs associated with the controlled trajectories  of the system started \new{at any point $x$, where controls are chosen according to the feedback strategy induced by the subsolution,} gives an upper bound on the value function at $x$. These a posteriori error bounds provide a way to assess the precision of the approximation.
The algorithm exploits a certain assumed convexity property of the control problem in order to generate a non-decreasing sequence of subsolutions. 
To be more precise, we consider problems whose Bellman operator, seen as an operator between function spaces, has the property of preserving convexity, that is, \new{for some set of convex functions the Bellman operator maps that set into itself. For problems of this type,} the value function is convex. \new{We give two examples of classes of convexity-preserving control problems, one with affine-linear dynamics  (\emph{linear-convex models}), the other with dynamics enjoying a certain monotonicity property (\emph{monotone-convex models}). Our assumption on the Bellman operator} implies that suitably constructed elementary ({hyperplane}) functions are subsolutions. The subsolutions produced by the algorithm are then piecewise affine functions defined as pointwise maxima of a growing family of hyperplanes. 

The scheme we present is inspired by \cite{F}, where a related procedure was proposed to solve a different class of problems. The algorithm consists of two main steps. First, starting from an initial point that can be fixed or sampled from a given distribution, a trajectory of the system is simulated by means of the currently stored subsolution. Second, the stored subsolution is updated along the generated trajectory. Under mild assumptions, the functions produced converge everywhere and locally uniformly to a convex subsolution of the optimal control problem. Furthermore, the subsolutions generated in this way converge to the value function at time zero on the support of the initial distribution \new{with probability one}. The subsolution of a control problem is less or equal than the value function, thus the algorithm has a global scope: not only it returns the value function at time zero, but it also gives true lower bounds on the whole state space. \new{Finally, the structure of our algorithm makes it well-suited for parallelization. This can be accomplished by sampling $M$ starting points and independently simulating $M$ trajectories of the system during the trajectory simulation step. Then, the new subsolution is obtained by simultaneously updating the previous subsolution along the $M$ trajectories. For more details see Remark \ref{rem:parallelization} below.}

\new{Our work differs from \cite{F} in several crucial aspects. First, \cite{F} considers problems whose Bellman operator maps \textit{globally} Lipschitz functions into globally Lipschitz functions. In particular, the value function is globally Lipschitz continuous. This is in contrast with our convex setting. In fact, the algorithm in \cite{F} contructs subsolutions as the maxima of cone-shaped elementary functions. Second, here we face the additional technical challenge of bounding the costs away from the origin. Third, we extend \cite{F} by considering a random initial condition.}

The idea of approximating convex functions with their supporting hyperplanes dates back to Kelley \cite{KE}, who was one of the first proponents of the cutting plane algorithm for convex optimization \new{(cf.\ \cite{bertsekas2015}). The cutting plane algorithm} deals with a problem of static optimization, while our approximation scheme crucially exploits the dynamic structure of the control problem.

\new{We now present a brief overview of known solution methods for deterministic and stochastic optimal control problems. Although our method applies to discrete-time problems, it is motivated by the solution of optimal control problems in continuous time. We thus view our problems as resulting from time discretization of the underlying controlled dynamics (also cf.\ the numerical experiments below). In comparing our approach to established methods, we therefore focus on the literature dealing with the numerical solution of continuous-time problems.

For a large class of continuous-time optimal control problems, the value function of the problem can be characterized through an infinitesimal version of the Dynamic Programming Principle as the unique viscosity solution of a first or second order partial differential equation, the so-called \emph{Hamilton-Jacobi-Bellman} (HJB); cf.\ \cite{flemingsoner06}. Having computed the value function also yields an optimal feedback or \emph{closed-loop} strategy. A wide range of methods have been developed to approximate the solutions of HJB equations numerically. One of the earlier references is \cite{capuzzo1989discrete}, where a monotone fully discrete scheme is proposed to approximate the viscosity solution of the HJB equation associated with a deterministic continuous-time control problem. Similar techniques have been applied to stochastic control problems as well, see \cite{barlesjakobsen2007} and the references therein. A large class of approximation algorithms is presented in \cite{kushner2001numerical}, collectively referred to as \emph{Markov Chain approximation}. The common idea underlying these methods is to approximate the controlled continuous-time process by a Markov Chain on a finite state space, and then to solve this simpler control problem, in particular through finite-state discrete-time dynamic programming. The approximating chain is parametrized by a parameter analogous to the step size of a finite grid, which regulates the precision of the approximation. Semi-Lagrangian schemes for HJB equations are based on a discretization of the underlying controlled characteristics \cite{falcone2013semi}; they can be given a probabilistic interpretation in terms of a Markov Chain approximation. Those methods are powerful for the numerical solution of problems with low-dimensional state space, but generally not applicable when the state space dimension is large \cite{grune2015numerical}. On the other hand, the HJB approach works under rather general assumptions. This is unlike our approach, which crucially exploits the specific structure of the control problem.

A rather different approach to the solution of dynamic optimal control problems is presented in \cite{mceneaney2006maxplus}. This method is based on the observation that the Bellman operator of a control problem is linear with respect to the operations of a suitably chosen algebra. In this so-called \emph{maxplus algebra}, summation between two elements is defined as their maximum, and multiplication as the usual summation. If the control problem is a maximization problem, the maxplus approach approximates the value function via a linear combination (in the maxplus algebra) of suitably chosen basis functions. Since the Bellman operator is maxplus-linear, the application of the Dynamic Programming Principle reduces to (maxplus) matrix-vector multiplication. This approach shares some features with our scheme. In both approaches the value function is approximated from below via a (maxplus-)linear combination of elementary functions, which are computed exploiting the Dynamic Programming Principle. However, the two approaches differ in one crucial aspect. The maxplus algorithms approximate from below the value function associated with a \emph{maximization} problem, while our scheme approximates from below a \emph{minimization} problem. The two approaches are not equivalent, because for minimization problems, the maxplus algebra is replaced by the minplus algebra, and consequently the value function is approximated via the minimum of elementary functions.

More recently, \cite{picarellireisinger} applies the semi-Lagrangian approach to solve an optimal problem over a finite time horizon with one-dimensional linear dynamics and where only the terminal costs are non-zero. The dual problem is a minimization problem with linear dynamics and convex terminal costs, and thus is compatible with our key assumption. The continuous-time problem is approximated by discretizing space as well as time, and by computing expectations with respect to the (one-dimensional) Gaussian noise via the Gauss-Hermite quadrature formula. Exploiting duality, the authors obtain, in addition to a one-sided \emph{a priori} error bound, two-sided \emph{a posteriori} error bounds for the approximate value function in terms of the discretization steps and the number of quadrature points. As the scheme relies on a full state space discretization, it would not be suitable for problems with high-dimensional state space.    

The value function of a discrete-time finite horizon problem can in principle be computed, thanks to the Dynamic Programming Principle, through backward induction starting from the terminal costs. Instead of discretizing the state space, one can propagate backwards an estimate of the value function at each step in time. Those estimates can be obtained by sampling support points in space and computing conditional expectations. In \cite{belomestny2010regression} several Monte Carlo-based schemes are proposed to obtain  estimates for the value function and compute conditional expectations. It is assumed that the kernel of the controlled process can be decomposed as the product of an uncontrolled kernel and a density representing the effect of the control on the dynamics. This particular structure allows to sample state points from an uncontrolled forward process and to obtain efficient global approximations for the value function. For a survey of Monte Carlo techniques in the context of dynamic programming, see the book \cite{powell2007approximate} and references therein.





Another succesful approach to the solution of continuous-time optimal control problems is based on Pontryagin's Maximum Principle (PMP). This produces the approximate optimal trajectory of the system, as well as an \emph{open-loop} optimal control $\mu^*(t)$. This is interpreted as a local solution to the optimal control problem.
In \cite{bismut1973conjugate}, linear backward stochastic differential equations (BSDEs) were introduced as the adjoint equations in the Pontryagin Maximum Principle of a stochastic control problem. Since then, a large literature has grown on the topic of approximating the solution of a general BSDE. Recently, \cite{briand2014simulation, geiss2016simulation} proposed a numerically-efficient algorithm to solve BSDEs based on Wiener chaos expansion. On the one hand, like the HJB approach, this approach also works under rather general assumptions. On the other hand, the PMP approach only provides a local solution to the control problem. This should be contrasted with our scheme, which provides double-sided global approximations for the value function.

In \textit{direct discretization} approaches, the problem is first converted into a static nonlinear optimization problem by discretizing time and by considering control strategies that are parametrized by a finite number of parameters. The problem is then solved using fast and reliable NLP solvers, see, e.g., \cite{andersson2019casadi,hall2021sequential} and references therein. Close to direct discretization is \cite{lasserre2008nonlinear}, where a class of nonlinear \textit{deterministic} optimal control problems is successively approximated from below via the so called LMI-relaxations. In this approach, the original control problem is replaced by an infinite-dimensional linear program  (LP) over two spaces of measures. If the dynamics and control constraints are polynomials, only the moments of the two measures appear in the LP formulation. Considering a truncation of the problem to a finite but growing number of moments gives a non-decreasing sequence of lower bounds on the optimal value of the LP. This approach is also \textit{local} since it provides an approximation of the value function and the optimal control policy starting at one point of the state space. Compared to the direct discretization approaches, which seem to be particularly powerful for deterministic problems, our method is clearly more limited in scope, as it only applies to (deterministic or stochastic) convex problems. Its main advantage lies in the construction of subsolutions, which yield sub-optimal feedback strategies on the entire state space. While direct discretization yields the optimal value, together with a (nearly) optimal open-loop control, at any fixed initial state, the subsolutions constructed in our approach starting from a fixed initial state provide sub-optimal feedback controls for arbitrary initial states. This is particularly useful for (small) perturbations of the original initial state. We illustrate this point at the end of Subsection \ref{sec:deterministic_numerical_experiments}.

More recently, the use of machine learning techniques, more specifically deep learning, has been proposed to solve high-dimensional optimal control problems. One of the crucial advantages of these methods is that deep neural networks are capable of representing high-dimensional functional data without a space discretization scheme. Another advantage is the wide availability of powerful neural network libraries such as \textsc{TensorFlow}. Deep learning algorithms have been applied in a variety of ways in the context of optimal control. In the discrete-time setting, \cite{fecamp2019risk,han2016deep} uses a state-of-the-art deep neural network to minimize the cost functional and generate a Markov feedback control policy. In \cite{bachouch2021deep,hure2021deep}, classical dynamic programming backward iteration schemes are combined with neural network numerical computations to find an optimal policy and the value function. For more details and a recent review of these approaches, see \cite{germain2021neural}. In the continuous-time setting, neural networks have been used for solving first-order Hamilton-Jacobi equations \cite{darbon2020overcoming,darbon2021some} corresponding to deterministic control problems and, more generally, nonlinear PDEs, see, e.g., \cite{weinan2021algorithms} for a review of recent results. Like our algorithm, these methods do not rely on a state-space discretization scheme. As compared to these extremely powerful and general deep learning approaches, we stress that, for our class of problems, we achieve useful results with limited computational resources. Our numerical tests are carried out on a small workstation using a single \mbox{CPU} while, for instance, in \cite{germain2021neural} numerical experiments are carried out using hundreds of \mbox{CPUs} and \mbox{GPUs} in parallel.
}


The rest of the paper is organized as follows. In Section \ref{sec:optimization_problem}, we introduce a class of discrete-time deterministic and stochastic optimal control problems \new{with convexity-preserving Bellman operator, and show that it includes monotone-convex and} linear-convex models. In Section \ref{sec:algorithm}, we present our algorithm in the stochastic setting and prove that it produces a non-decreasing sequence of subsolutions. In Section~\ref{sec:convergence_deterministic_systems}, we focus on the deterministic setting and show that \new{for any fixed initial state} the subsolutions generated by the algorithm converge to the value function at time zero. In this setting, we also obtain that the controls (resp.\ trajectories) generated by the algorithm converge \new{along subsequences} to the optimal controls (resp.\ optimal trajectories). In Section~\ref{sec:convergence_stochastic_systems}, we extend the proof of convergence to the case of stochastic dynamics, allowing for random initial states. In Section~\ref{sec:numerical_experiments}, we present the results of numerical experiments on deterministic and stochastic linear-convex control problems.

\section{The optimization problem}\label{sec:optimization_problem}

The control problems we consider are a subclass of the semicontinuous Borel models studied in \cite{bertsekasshreve96}. The dynamics are given by controlled time inhomogeneous Markov chains with {state space} $\mathcal{X}$, {disturbance space} $\mathcal{Y}$, and {action space} $\Gamma$. For fixed dimensions $d, d_{1}, d_{2}\in \mathbb{N}$, we set $\mathcal{X}\doteq \mathbb{R}^{d}$, $\mathcal{Y}\doteq \mathbb{R}^{d_{1}}$, and choose $\Gamma$ as a non-empty measurable subset of $\mathbb{R}^{d_{2}}$. The evolution of the system is determined by the so-called \emph{system function}, which we take to be a Borel-measurable function $\Psi\!: \mathbb{N}_{0}\times \mathcal{X}\times \Gamma\times \mathcal{Y} \rightarrow \mathcal{X}$, \new{where $\mathbb N_0$ denotes the set of non-negative integers.}

Let $\mu$ be a probability measure on the Borel $\sigma$-algebra of $\mathcal{Y}$; $\mu$ will be called \emph{noise distribution}. We will refer to any $\mathcal{Y}$-valued random variable that has distribution $\mu$ as {disturbance} or {noise variable}. Let $\mathcal{M}$ denote the set of all non-randomized {Markov feedback control policies}, that is,
\begin{align*}%
  \mathcal{M}\doteq \{u : \mathbb{N}_{0}\times \mathcal{X} \rightarrow \Gamma \text{ such that $u$ is Borel measurable}\}.
\end{align*}%
Fix a complete probability space $(\Omega,\mathcal{F},\Prb)$ \new{which supports a sequence $(\xi_{j})_{j\in\mathbb{N}}$ of independent and identically distributed (i.i.d.) noise variables. Given any Markov control policy $u \in \mathcal{M}$, initial state $x_{0}\in \mathcal{X}$,} and initial time $j_{0} \in \mathbb{N}_{0}$, the corresponding {state sequence} is recursively defined, for each $\omega \in \Omega$, by
\begin{align} \label{eq:state_recursion}%
   X_{j_{0}}(\omega)\doteq x_{0},\quad X_{j+1}(\omega)\doteq \Psi(j,X_{j}(\omega),u(j,X_{j}(\omega)),\xi_{j+1}(\omega)),\, j \geq j_{0}.&
\end{align}%
Performance is measured in terms of expected costs over a finite deterministic {time horizon}, denoted by $N \in \mathbb{N}$. Let $f\!:\mathbb{N}_{0}\times \mathcal{X}\times \Gamma \rightarrow \mathbb{R}$, $F\!: \mathcal{X} \rightarrow \mathbb{R}$ be measurable functions, denoting respectively the {running costs} and the {terminal costs}. We will refer to $f$, $F$ as \emph{cost coefficients}. The {cost functional} $J$ is then defined as
\begin{equation} \label{eq:expected_costs}%
  J(j_{0},x_0,u)\doteq \Mean\left[\sum_{j=j_{0}}^{N-1} f\bigl(j,X_j,u(j,X_{j})\bigr) + F(X_N)\right]\!,
\end{equation}%
where $(X_j)_{j\geq j_{0}}$ is the state sequence generated according to \eqref{eq:state_recursion} with control policy $u$ and $X_{j_{0}} \doteq x_{0}$. If $f$, $F$ are bounded from below, then $J$ is well defined as a mapping $\{0,\ldots,N\}\times \mathcal{X}\times \mathcal{M} \rightarrow \mathbb{R}\cup\{\infty\}$. Notice that $J$ depends on the noise variables only through their distribution and does not depend on the particular choice of the underlying probability space.

The \emph{value function} of the control problem is defined as
\begin{equation} \label{eq:value_function}%
  V(j_{0},x_{0})\doteq \inf_{u\in\mathcal{M}} J(j_{0},x_{0},u),\quad (j_{0},x_{0})\in \{0,\ldots,N\}\times \mathcal{X}.
\end{equation}%
If $f$ and $F$ are bounded from below, then $V$ is well defined as a mapping $\{0,\ldots,N\}\times \mathcal{X} \rightarrow \mathbb{R}\cup\{\infty\}$. The value function $V$ is often understood as the \emph{solution} of the control problem. 

For each $j\in\mathbb N_0$, the \emph{one-step Bellman operator} $\mathcal{L}_{j}$ is a functional that acts on the space of functions $g\!: \mathcal{X} \rightarrow \mathbb{R}\cup\{\infty\}$ as
\begin{align} \label{eq:bellman_operator}%
  & \mathcal{L}_{j}(g)(x)\doteq \inf_{\gamma\in\Gamma} \left\{f(j,x,\gamma) + \int_{\mathcal{Y}} g\bigl(\Psi(j,x,\gamma,y)\bigr) \mu(\mathrm dy)\right\}, & & x \in \mathcal{X},&
\end{align}%
provided the right hand side is well defined. Note that, crucially, the Bellman operator is monotone: \new{If $g$, $\tilde{g}$ are functions such that $\tilde{g} \geq g$ (in the sense that $\tilde{g}(x) \geq g(x)$ for every $x\in \mathcal{X}$)} and $\mathcal{L}_{j}(g)$, $\mathcal{L}_{j}(\tilde{g})$ are well defined, then $ \mathcal{L}_{j}(\tilde{g})\geq \mathcal{L}_{j}(g)$.

We make the following assumptions. \new{With $\mathfrak{C}$ a set of convex functions $\mathcal{X} \rightarrow [0,\infty)$:
\begin{hypenv}%
  \item\label{hyp:system_function_continuous} The system function $\Psi$ is continuous.

  \item\label{hyp:cost_coefficients} The cost coefficients $f$, $F$ are non-negative and continuous. 
  
  \item\label{hyp:control_space_compact} The space of control actions $\Gamma$ is compact.
  
  \item \label{hyp:terminal_cost_convex} The terminal cost coefficient $F$ is in $\mathfrak{C}$.
  
  \item\label{hyp:bellman_operator_preserves_convexity} The Bellman operator preserves the set of convex functions $\mathfrak{C}$: If $g\in \mathfrak{C}$, then $\mathcal{L}_j(g)\in \mathfrak{C}$ for every $j\in \mathbb{N}_{0}$.

\end{hypenv}

Here, $\mathfrak{C}$, the set of non-negative convex functions appearing in \hypref{hyp:terminal_cost_convex} and \hypref{hyp:bellman_operator_preserves_convexity}, is supposed to satisfy the following structure assumptions:
\begin{enumerate}[(i)]
  \item the constant function equal to zero is in $\mathfrak{C}$;
  
  \item if $g, \tilde{g}\in \mathfrak{C}$, then $g \vee \tilde{g} \in \mathfrak{C}$, where $(g \vee \tilde{g})(x)\doteq \max\{g(x), \tilde{g}(x)\}$, $x\in \mathcal{X}$;
  
  \item if $g \in \mathfrak{C}$, $x \in \mathcal{X}$, and $v \in \mathbb{R}^{d}$ is such that $g(y) - g(x) \geq v\cdot (y-x)$ for all $y\in \mathcal{X}$, then $y \mapsto \max\{0, g(x) + v\cdot(y-x)\}$ is in $\mathfrak{C}$.
  
\end{enumerate}

}

\new{Assumptions \hypref{hyp:system_function_continuous}--\hypref{hyp:control_space_compact}} are standard assumptions in the optimal control literature, see for example \cite[Condition 3.3.2, p.~27]{hernandezlermalasserre96}, and \cite[Proposition C.4, p.~176]{hernandezlermalasserre96}. \new{In \hypref{hyp:cost_coefficients}, semi-continuity of the cost coefficients would be sufficient; we exploit their continuity only in the proof of Theorem~\ref{prop:convergence_deterministic_dynamics}.}

Assumption~\hypref{hyp:bellman_operator_preserves_convexity}, on the other hand, is the key ingredient for our scheme to construct subsolutions. \new{Together with Assumption~\hypref{hyp:terminal_cost_convex}, it will imply that the value function is convex. Below, we provide two classes of models where the Bellman operator preserves a set of convex functions, one with linear dynamics and convex costs (linear-convex models), the other with non-linear monotone-convex dynamics and convex costs (monotone-convex models). In the former case, \hypref{hyp:bellman_operator_preserves_convexity} is satisfied with $\mathfrak{C}$ chosen as the set of all non-negative convex functions (provided the noise distribution $\mu$ has bounded support). In the latter case, \hypref{hyp:bellman_operator_preserves_convexity} holds with $\mathfrak{C}$ as the set of all non-negative non-decreasing convex functions.}

We begin by collecting some standard consequences of Assumptions \hypref{hyp:system_function_continuous}--\hypref{hyp:control_space_compact}. If $g : \mathcal{X} \rightarrow \mathbb{R}\cup\{\infty\}$ is lower semicontinuous and bounded from below, then $\mathcal{L}_{j}(g) : \mathcal{X} \rightarrow \mathbb{R}\cup\{\infty\}$ is also lower semicontinuous and bounded from below. 
Assumptions \hypref{hyp:system_function_continuous}--\hypref{hyp:control_space_compact} also guarantee that the Principle of Dynamic Programming holds and that optimal Markov feedback policies exist.
\begin{lemma}[Dynamic Programming, \cite{bertsekasshreve96}] \label{lem:principle_dynamic_programming}%
  \new{Assume \hypref{hyp:system_function_continuous}--\hypref{hyp:control_space_compact}.} Then $V$ is lower semicontinuous, \new{non-negative}, and for all $x \in \mathcal{X}$, $j \in \{0,\ldots,N\!-\!1\}$,
  \begin{align*}%
      V(j,x) = \mathcal{L}_{j}(V(j\!+\!1,\cdot))(x).
  \end{align*}%
Moreover, an optimal Markov feedback control policy exists, that is, there is $u \in \mathcal{M}$ such that $V(j,x) = J(j,x,u(j,x))$ for all $x \in \mathcal{X}$, $j \in \{0,\ldots,N\!-\!1\}$.
\end{lemma}%
\new{In conjuction with Assumptions \hypref{hyp:terminal_cost_convex} and \hypref{hyp:bellman_operator_preserves_convexity}, Lemma~\ref{lem:principle_dynamic_programming} implies that the value function $V$ is convex and finite-valued with values in $[0,\infty)$, hence also locally Lipschitz continuous; see, for instance, Theorem~B.3.1.2 in \cite[p.\,105]{HUL2001}.
}

The class of control policies can be enlarged in different ways without changing the value function. In Section~\ref{sec:convergence_stochastic_systems} we will consider the following weak formulation of our control problems; cf.~Definition~2.4.2 in \cite[p.\,64]{yongzhou99} in the context of continuous-time models.
\begin{definition}\label{def:stochastic_control_policy}
A \emph{stochastic control policy} is a quadruple\\ 
$((\Omega,\mathcal{F},\mathbf{P}),(\mathcal{F}_{j})_{j\in\mathbb{N}_{0}},(\xi_{j})_{j\in\mathbb{N}},(\gamma_{j})_{j\in\mathbb{N}_{0}})$ such that
\begin{enumerate}[(i)]%
  \item $(\Omega,\mathcal{F},\mathbf{P})$ is a complete probability space;
  
  \item $(\mathcal{F}_j)_{j\in\mathbb{N}_{0}}$ is a filtration in $\mathcal{F}$;

  \item $(\xi_{j})_{j\in\mathbb{N}}$ is an independent sequence of noise variables such that, for every $j\in \mathbb{N}$, $\xi_{j}$ is $\mathcal{F}_{j}$-measurable and independent of $\mathcal{F}_{j-1}$;

  \item $(\gamma_{j})_{j\in\mathbb N_0}$ is a sequence of $\Gamma$-valued random variables adapted to the filtration $(\mathcal{F}_{j})$ (i.e., $\gamma_{j}$ is $\mathcal{F}_{j}$-measurable for every $j\in \mathbb{N}_{0}$).

\end{enumerate}%
\end{definition}%
Let $\mathcal{R}$ denote the set of all stochastic control policies. We will suppress the dependence of $(\gamma_{j})_{j\in\mathbb N_0}$ on the stochastic basis; thus, we will write $(\gamma_{j})\in \mathcal{R}$ instead of $((\Omega,\mathcal{F},\Prb), (\mathcal{F}_{j})_{j\in\mathbb{N}_{0}}, (\xi_{j})_{j\in\mathbb{N}}, (\gamma_{j})_{j\in\mathbb{N}_{0}}) \in \mathcal{R}$.

Given a stochastic control policy $(\gamma_{j})\in \mathcal{R}$, an initial state $x_{0}$, and an initial time $j_{0} \in \mathbb{N}_{0}$, the corresponding state sequence is recursively defined, for each $\omega\in \Omega$, by
\begin{align} \label{eq:state_recursion_weak}%
  & X_{j_{0}}(\omega)\doteq x_{0},& & X_{j+1}(\omega)\doteq \Psi\bigl(j,X_{j}(\omega),\gamma_{j}(\omega),\xi_{j+1}(\omega)\bigr),& &j \geq j_{0}.&
\end{align}%
The costs associated with such a state sequence and control policy are given by
\begin{equation} \label{eq:cost_functional_weak}%
  \hat{J}\bigl(j_{0},x_{0},(\gamma_{j})\bigr)\doteq \Mean\left[\sum_{j=j_{0}}^{N-1} f(j,X_j,\gamma_{j}) + F(X_N)\right],
\end{equation}%
where expectation is taken with respect to the probability measure of the stochastic basis coming with $(\gamma_{j})$, and $X_{j_{0}}(\omega)\doteq x_0$. The \emph{value function} in the weak formulation is defined by
\begin{equation} \label{eq:value_function_weak}%
  \hat{V}(j_{0},x_{0})\doteq \inf_{(\gamma_{j})\in\mathcal{R}} \hat{J}\bigl(j_{0},x_{0},(\gamma_{j})\bigr), \quad (j_{0},x_{0})\in \{0,\ldots,N\}\times \mathcal{X}.
\end{equation}%

Under \new{Assumptions \hypref{hyp:system_function_continuous}--\hypref{hyp:control_space_compact}}, the value function $V$ defined in \eqref{eq:value_function} over the class of Markov feedback control policies with fixed stochastic basis coincides with $\hat{V}$, the value function defined in \eqref{eq:value_function_weak} over the class of stochastic control policies with varying stochastic basis. Results of this type are standard; \new{for details see, for instance, Lemma~2 in \cite{F}.}
\begin{lemma} \label{lem:value_functions_strong_weak_equal}
  \new{Assume \hypref{hyp:system_function_continuous}--\hypref{hyp:control_space_compact}}. Then $V = \hat{V}$.
\end{lemma}
The notion of subsolution is central to our algorithm.

\begin{definition}
A function $w$ from $\{0,\ldots,N\}\times \mathcal{X}$ to $\mathbb{R}\cup \{\infty\}$ is a  \emph{subsolution} of the control problem if
\begin{enumerate}[(i)]
  \item $w$ is lower semicontinuous and bounded from below;
  
  \item $w(T,x) \leq F(x)$ for all $x\in \mathcal{X}$;
  
  \item for every $j\in\{0,\ldots,N\!-\!1\}$, $w(j,x)\leq \mathcal{L}_{j}(w(j+1,.))(x)$ for all $x\in \mathcal{X}$.
\end{enumerate}
\end{definition}
We observe that a subsolution never exceeds the value function: If $w$ is a subsolution, then $w(j,x)\leq V(j,x)$ for all $(j,x)\in \{0,\ldots,N\!-\!1\}\times \mathcal{X}$. This is a consequence of Lemma~\ref{lem:principle_dynamic_programming} and the monotonicity of the Bellman operator.
\begin{exmpl}[Example: Linear-convex models] \label{ex:linear_convex_models}
Let $A$, $B$, $C$ be real-valued matrices of dimensions $d\times d$, $d\times d_{2}$, and $d\times d_{1}$, respectively. Let $\Gamma \subset \mathbb{R}^{d_{2}}$ be compact and \emph{convex}. Consider the system function
\begin{align*}
\Psi(j,x,\gamma,y)\doteq Ax + B\gamma + C y,\quad (j,x,\gamma,y)\in \mathbb{N}_{0}\times \mathcal{X}\times \Gamma\times \mathcal{Y}.
\end{align*}%
\new{Choose non-negative cost coefficients $f$, $F$ such} that $(x,\gamma)\mapsto f(j,x,\gamma)$, $j\in \mathbb{N}_{0}$,  and $x\mapsto F(x)$ are convex. We emphasize that convexity of $f$ in the state variable alone is not enough for our purposes. Choose any probability measure $\mu$ \new{with bounded support} on the Borel sets of $\mathcal{Y} = \mathbb{R}^{d_{2}}$ as noise distribution. This is a \emph{linear-convex Markov control problem}. \new{Choose $\mathfrak{C}$ as the set of all non-negative convex functions on $\mathcal{X}$. Then} Assumptions\hyprefallbutlast\ are satisfied. To check \hypref{hyp:bellman_operator_preserves_convexity} we \new{prove that the function $z\mapsto\mathcal L_j(g)(z)$ is convex whenever $g$ is a convex function. Let $x, y \in\mathcal X$, and let $\lambda_1, \lambda_2 \in [0,1]$ be such that $\lambda_1 + \lambda_2 = 1$. The Bellman operator computed in the convex combination $\lambda_1 x + \lambda_2 y$ is}
\begin{align*}%
\mathcal{L}_j(g)&\left(\lambda_1 x + \lambda_2 y \right) \\
 &=\inf_{\gamma\in \Gamma}\left\{f(j,\lambda_1 x + \lambda_2 y,\gamma)+\int_{\mathcal{Y}} g\left(\lambda_1 Ax+ \lambda_2 Ay + B\gamma + Cz\right) \mu (\mathrm d z)\right\}.
\end{align*}%
Now choose $\gamma_x,\gamma_y\in\Gamma$ such that they minimize respectively $\mathcal{L}_j(g)(x)$ and $\mathcal{L}_j(g)(y)$. Set $\tilde{\gamma}=\lambda_1 \gamma_x + \lambda_2 \gamma_y$. \new{Since $\Gamma$ is a convex set, $\tilde\gamma\in \Gamma$.} Then, bound $\mathcal L_j(g)$ as
\begin{align*}%
\mathcal{L}_j(g)(\lambda_1x + \lambda_2 y) &\leq f(j,\lambda_1x+\lambda_2 y,\tilde{\gamma})+\int_{\mathcal{Y}} g\left(\lambda_1 Ax+\lambda_2 Ay + B\tilde{\gamma}+ Cz\right) \mu(\mathrm d z) \\
&\leq \lambda_1f(j,x,\gamma_x)+\lambda_2f(j,y,\gamma_y)+\lambda_1 \int_{\mathcal{Y}} g\left(Ax+B\gamma_x+Cz\right) \mu (\mathrm d z)\\
&\quad+\lambda_2\int_{\mathcal{Y}} g\left(Ay+B\gamma_y + Cz\right) \mu (\mathrm d z) \\
&= \lambda_1\mathcal{L}_j(g)(x)+\lambda_2\mathcal{L}_j(\phi)(y).
\end{align*}%
In the second inequality above, we used the convexity of $g$ and of $f$ in both the state and control variable. The last equality holds by definition of $\gamma_x$ and $\gamma_y$. Thus Assumption \hypref{hyp:bellman_operator_preserves_convexity} holds for this class of problems.
\end{exmpl}

\new{
\begin{exmpl}[Example: Monotone-convex models]\label{ex:monotone_convex_models}
Recall that $\mathcal{X} = \mathbb{R}^{d}$, $\mathcal{Y} = \mathbb{R}^{d_1}$. Let $\Gamma \subset \mathbb{R}^{d_{2}}$ be compact and \emph{convex}. Let the system function $\Psi$ be of the form
\[
  \Psi(j,x,\gamma,y) = \tilde{\Psi}(x,y) + B(y) \gamma,
\]
where $B(y)$, $y\in \mathcal{Y}$, are $d\times d_{2}$-matrices such that $y \mapsto B(y)$ is bounded continuous, while $\tilde{\Psi} = (\tilde{\Psi}_{1},\ldots, \tilde{\Psi}_{d})$ is a continuous function $\mathcal{X}\times \mathcal{Y} \rightarrow \mathcal{X}$ such that, for every $i\in \{1,\ldots,d\}$, every $y\in \mathcal{Y}$, $x\mapsto \tilde{\Psi}_{i}(x,y)$ is convex and \emph{non-decreasing}. Here, a mapping $g\!: \mathbb{R}^{d} \rightarrow \mathbb{R}$ is non-decreasing if $g(x) \leq g(\tilde{x})$ whenever $x \leq \tilde{x}$ with respect to the component-wise (partial) order on $\mathbb{R}^{d}$. For example, $g(x) = \sum_{k=1}^{M} \phi_{k}\left(v^{(k)}\cdot x \right)$ is convex and non-decreasing if $v^{(k)} \in [0,\infty)^{d}$ and $\phi_{k}\colon \mathbb{R} \rightarrow \mathbb{R}$ is convex and non-decreasing for every $k\in \{1,\ldots,M\}$.

Choose non-negative cost coefficients such that $F$ is convex and non-decreasing, $(x,\gamma)\mapsto f(j,x,\gamma)$ is convex for all $j\in \mathbb{N}_{0}$, and $x\mapsto f(j,x,\gamma)$ is non-decreasing for all $j\in \mathbb{N}_{0}$, $\gamma \in \Gamma$. An example of such a function $f$ is $f(j,x,\gamma) = \max\{0,v\cdot x\} + |\gamma|^{2}$ for some $v\in \mathbb{R}^{d}$ with $0 \leq v$. Finally, choose a probability measure $\mu$ with bounded support on the Borel sets of $\mathcal{Y}$ as noise distribution. Choose $\mathfrak{C}$ as the set of all non-negative non-decreasing convex functions on $\mathcal{X}$. Then \hyprefallbutlast\ as well as the structure assumptions on $\mathfrak{C}$ are satisfied.

We are going to show that the function $x\mapsto\mathcal L_j(g)(x)$ is convex non-negative and non-decreasing whenever $g\!: \mathcal{X} \rightarrow \mathbb{R}$ is convex non-negative and non-decreasing. Let $g$ be such a function. That $L_j(g)$ is non-negative is clear from the monotonicity of the Bellman operator and the non-negativity of the running costs $f$. Let $x, \tilde{x} \in \mathbb{R}^{d}$ with $x \leq \tilde{x}$. Then, for all $\gamma \in \Gamma$,
\[
  \int_{\mathcal{Y}} g\left(\tilde{\Psi}(x,y) + B(y)\gamma \right) \mu(\mathrm d y) \leq \int_{\mathcal{Y}} g\left(\tilde{\Psi}(\tilde{x},y) + B(y)\gamma \right) \mu(\mathrm d y)  
\]
by the monotonicity of the integral and since $g$ is non-decreasing and $\tilde{\Psi}(x,y) \leq \tilde{\Psi}(\tilde{x},y)$ for all $y\in \mathcal{Y}$ as the component functions $\tilde{\Psi}_{i}$ are non-decreasing in their first argument. Since $x\mapsto f(j,x,\gamma)$ is non-decreasing for all $\gamma \in \Gamma$, we find that $L_j(g)(x) \leq L_j(g)(\tilde{x})$. Thus, $L_j(g)$ is non-decreasing.

Now, check that the mapping $(x,\gamma) \mapsto \int_{\mathcal{Y}} g\left(\tilde{\Psi}(x,y) + B(y)\gamma \right) \mu(\mathrm d y)$ is convex on $\mathcal{X}\times \Gamma$. Indeed, with $x, \tilde{x} \in \mathbb{R}^{d}$, $\gamma, \tilde{\gamma} \in \Gamma$, $\lambda_1, \lambda_2 \in [0,1]$ such that $\lambda_1 + \lambda_2 = 1$, we have
\[
  \tilde{\Psi}\left(\lambda_1 x + \lambda_2\tilde{x} ,y \right) \leq \lambda_1 \tilde{\Psi}(x,y) + \lambda_2 \tilde{\Psi}(\tilde{x},y)
\]
for every $y\in \mathcal{Y}$ by the convexity of the component functions $\tilde{\Psi}_{i}$ in their first argument. Since $g$ is non-decreasing and convex, we find that
\begin{align*}
  \int_{\mathcal{Y}} & g\left(\tilde{\Psi}\left(\lambda_1 x + \lambda_2\tilde{x},y\right) + B(y)(\lambda_1\gamma + \lambda_2 \tilde{\gamma}) \right) \mu(\mathrm d y) \\
  &\leq \int_{\mathcal{Y}} g\left( \lambda_1 \tilde{\Psi}(x,y) + \lambda_2 \tilde{\Psi}(\tilde{x},y) + \lambda_1 B(y)\gamma + \lambda_2 B(y)\tilde{\gamma} \right) \mu(\mathrm d y) \\
  &\leq \lambda_1 \int_{\mathcal{Y}} g\left(  \tilde{\Psi}(x,y) + B(y)\gamma \right) \mu(dy)  + \lambda_2 \int_{\mathcal{Y}} g\left(\tilde{\Psi}(\tilde{x},y) + B(y)\tilde{\gamma} \right) \mu(\mathrm d y).
\end{align*}
By choice of the running costs $f$, it follows that the mapping
\[
  \mathcal{X}\times \Gamma \ni (x,\gamma) \mapsto f(j,x,\gamma) + \int_{\mathcal{Y}} g\left(\tilde{\Psi}(x,y) + B(y)\gamma \right) \mu(\mathrm d y)
\]
is convex. It is also finite-valued, thanks to the boundedness of the support of $\mu$, and non-negative. Define a function $\phi\!: \mathbb{R}^{d} \times \mathbb{R}^{d_2} \rightarrow [0,\infty]$ by
\[
  \phi(x,\gamma) \doteq \begin{cases}
  f(j,x,\gamma) + \int_{\mathcal{Y}} g\left(\tilde{\Psi}(x,y) + B(y)\gamma \right) \mu(\mathrm d y) &\text{if } \gamma \in \Gamma, \\
  \infty &\text{else.}
  \end{cases}
\]
Then $\phi$ is convex. By Corollary B.2.4.5 in \cite[p.\,98]{HUL2001}, we find that the marginal function $\hat{\phi}$ given by
\[
  \hat{\phi}(x) \doteq \inf\left\{ \phi(x,\gamma) : \gamma \in \mathbb{R}^{d_2} \right\}
\]
is convex on $\mathbb{R}^{d} = \mathcal{X}$. By construction, $\phi(x,\gamma) < \infty$ if and only if $\gamma \in \Gamma$, hence
\begin{multline*}
  \hat{\phi}(x) = \inf\left\{ \phi(x,\gamma) : \gamma \in \Gamma \right\} \\
  = \inf_{\gamma\in \Gamma} \left\{ f(j,x,\gamma) + \int_{\mathcal{Y}} g\left(\tilde{\Psi}(x,y) + B(y)\gamma \right) \mu(\mathrm d y) \right\} = L_j(g)(x),
\end{multline*}
which shows that $x\mapsto L_j(g)(x)$ is convex.

\end{exmpl}

In the above examples, dynamics are time-homogeneous, but the matrices $A$, $B$, $C$ in Example~\ref{ex:linear_convex_models} and the functions $\tilde{\Psi}$, $B(\cdot)$ in Example~\ref{ex:monotone_convex_models} may clearly be taken time-dependent. To allow for a noise distribution $\mu$ with unbounded support (and, in Example~\ref{ex:monotone_convex_models}, an only locally bounded function $B$) one has to require an appropriate growth condition of the convex functions in $\mathfrak{C}$ so that the Bellman operator yields finite-valued functions. For instance, in the case of Example~\ref{ex:linear_convex_models}, if $\mu$ is Gaussian, then $\mathfrak{C}$ may consist of all non-negative convex functions of sub-polynomial growth.

}

\section{The algorithm}\label{sec:algorithm}

In this section we present our algorithm. We assume that at initialization a subsolution is known. The scheme then iterates two steps: the \emph{trajectory simulation} step and the \emph{subsolution updating} step. 
During the simulation step the currently stored subsolution is used to generate a trajectory of the system which, 
roughly speaking, minimizes the subsolution at each time step. During the updating step, the algorithm constructs a new elementary (hyperplane) subsolution. This is possible because the control problem satisfies Assumption \hypref{hyp:bellman_operator_preserves_convexity}. Finally, the algorithm produces a new subsolution by taking at each time the maximum of the hyperplane and the current subsolution. By definition, the sequence of subsolutions thus produced is non-decreasing. We emphasize a subtle point: the update of the subsolution is efficient since it is performed along the trajectories computed in the first step, which are (close to) local minima for the previous subsolution. 

We will denote the hyperplane in  $\mathbb{R}^d\times \mathbb{R}$ which contains the point $(\bar{x},v)$ and with slope $p$ as
\begin{align*}%
\text{Hyp}:(\mathbb{R}^d\times \mathbb{R} \times \mathbb{R}^d)\times \mathbb{R}^d\rightarrow \mathbb{R},\,\qquad \text{Hyp}(\bar{x},v,p)(x)\eqdef v+p\cdot (x-\bar{x}).
\end{align*}%
For a function $g:\mathbb R^d\rightarrow \mathbb R$, we denote by $D^-g(x)$ the \new{\emph{subdifferential}} of $g$ at the point $x$. By definition, this is the set \new{of all \emph{subgradients} of $g$ at $x$, given by}
\begin{align*}%
D^-g(x)\eqdef \{v\in\mathbb R^n: g(y) - g(x) \geq v\cdot (y-x)\quad\forall y\in\mathbb R^n\}.
\end{align*}%
Recall that for any convex function $g:U\rightarrow\mathbb R$ defined on a convex subset $U\subseteq\mathbb R^d$, the \new{subdifferential} $D^-g(x)$ of $g$ is non-empty \new{and compact} at each point $x\in U$.

Let $w^{(0)}$ be the given starting subsolution. 
Without loss of generality, we will assume that the function $x\mapsto w^{(0)}(j,x)$ is convex for every $j$. Note that taking $w^{(0)}\equiv 0$ and $w^{(0)}(T,\cdot)=F(\cdot)$, we get a subsolution $w^{(0)}$ that is also convex. \new{Let $(\Omega, \mathcal{F},\mathbf{P})$ be a probability space carrying a family $(\xi_j^{(n)})_{j\in\mathbb N, n\in \mathbb{N}_0}$ of independent noise variables (hence i.i.d.\ random variables with common law $\mu$) and a sequence $(\zeta^{(n)})_{n\in\mathbb N_0}$ of i.i.d.\ random variables with common law $\nu$, independent from $(\xi_j^{(n)})_{j\in\mathbb N, n\in \mathbb{N}_0}$. For consistency, we always indicate the iteration number of the algorithm with lowercase $n$. Thus, $\zeta^{(n)}$ represents the initial condition of the dynamics at the $n$-th iteration of the algorithm. An important special case is when $\nu=\delta_{x_0}$, where $\delta_{x_0}$ denotes the Dirac measure centered in $x_0\in\mathcal X$. The algorithm is initialized by sampling $\omega\in\Omega$. Then $\zeta^{(n)}(\omega)$ represents the initial state and $(\xi_j^{(n)}(\omega))_{j\in\mathbb N}$ the realizations of the noise for the dynamics at the $n$-th iteration of the algorithm. }
Algorithm \ref{alg:construct_subsolutions} below presents the procedure for constructing the new subsolution $w^{(n+1)}$ starting from the previous subsolution $w^{(n)}$ at a generic step $n+1$, with $n\geq0$.
\begin{algorithm}[!ht]
\caption{Construct subsolution $w^{(n+1)}$ from $w^{(n)}$}
\label{alg:construct_subsolutions}
\begin{algorithmic}
\STATE{\emph{Trajectory simulation}\\ }
\STATE{Set $X_0^n\eqdef \zeta^{(n)}(\omega)$.}
\FOR{$j=0,1,2,\ldots, T-1$}
\STATE{Select a control
  \begin{align*}
  \gamma_j^n(\omega)\in \text{argmin}_{\gamma\in\Gamma}\left\{f(j,X_j^n(\omega),\gamma)+\int_{\mathcal{Y}} w_{\omega}^{(n)}\left(j+1,\Psi(j,X_j^n(\omega),\gamma,y\right) \mu(\mathrm d y)\right\}.
  \end{align*}
  }
\vspace{-1em}
\STATE{Set
  $
  X^n_{j+1}(\omega)\eqdef \Psi\bigl(j,X_j^n(\omega),\gamma_j^n(\omega),\xi^{(n)}_{j+1}(\omega) \bigr).
  $
  }
\ENDFOR
\STATE{\emph{Subsolution updating}
}
\STATE{Select a subgradient $p^{n}_{T} \in D^-F(X_T^n(\omega))$. Set
\begin{align*}%
w_{\omega}^{(n+1)}(T,\cdot)\eqdef \max \bigl\{w_{\omega}^{(n)}(T,\cdot), \text{Hyp}(X_T^n(\omega),F(X_T^n(\omega)),p^{n}_{T})(\cdot) \bigr\}.
\end{align*}%
}
\vspace{-1.5em}
\FOR{$j=T-1,\ldots, 1,0$}
\STATE{Let
\begin{align*}
v_j^n\eqdef \mathcal{L}_j(w_{\omega}^{(n+1)}(j+1,\cdot))(X_j^n(\omega)),\quad
p_j^n\in D^-_x\mathcal{L}_j(w_{\omega}^{(n+1)}(j+1,\cdot))(X_j^n(\omega)).
\end{align*}%
}
\vspace{-1em}
\STATE{Set $w_{\omega}^{(n+1)}(j,\cdot)\eqdef \max \bigl\{w_{\omega}^{(n)}(j,\cdot), \text{Hyp}(X_j^n(\omega),v_j^n,p_j^n)(\cdot) \bigr\}$.}
\ENDFOR

\RETURN $w^{(n+1)}$
\end{algorithmic}
\end{algorithm}

\new{
\begin{rem}[Weak control problem]
The control policy generated at each step $n$ of the algorithm crucially depends on the realization of the noise random variables at all previous steps. Therefore, the control policies  generated at different steps of the algorithm are adapted to different filtrations. This is why we have introduced a weak formulation of our control problem. In fact, the control policy $(\gamma_j^n)_{j\in\mathbb N_0}$ generated by our algorithm at each step $n$ is a stochastic control policy in the sense of Definition \ref{def:stochastic_control_policy}, and thanks to Lemma \ref{lem:value_functions_strong_weak_equal}, our algorithm does indeed produce approximations of the original value function \eqref{eq:value_function}.
\end{rem}
}
\begin{rem}[Choice of controls]
For every measurable function $g(\cdot)$ the function $x\rightarrow \text{argmin}_{\gamma\in\Gamma}\{f(j,x,\gamma)+\int_{\mathcal{Y}} g(\Psi(j,x,\gamma,y)\mu (\mathrm d y)\}$ is a set-valued function. \new{Assumptions \hypref{hyp:system_function_continuous}--\hypref{hyp:control_space_compact}} guarantee that the so-called \emph{measurable selection condition} holds. 
The measurable selection condition states that for each $j=1,\ldots, N$ and for each measurable function $g(\cdot)$ there exists a Borel-measurable function $s:\mathbb{R}^d\rightarrow \Gamma$ such that: 
\begin{align*}
\min_{\gamma\in\Gamma}\big\{f(j,x,\gamma)+& \int_{\mathcal{Y}} g\left(\Psi(j,x,\gamma,y\right) \mu (\mathrm d y)\big\}\\
&=f(j,x,s(x))+\int_{\mathcal{Y}} g\left(\Psi(j,x,s(x),y)\right)\mu (\mathrm d y).
\end{align*}
Measurable selection is a crucial ingredient for the proof of the Dynamic Programming Principle. Moreover, this condition implies that the trajectory simulation step of our algorithm is a measurable procedure.
\end{rem}

\begin{rem}[Parallelization] \label{rem:parallelization}
The trajectory simulation and updating steps are well-suited for parallelization as follows. Take a non-degenerate starting distribution $\nu$. In the trajectory simulation step, sample ${M\gg 1}$ independent starting points $x_1,\ldots,x_M$ according to the distribution $\nu$ and for each of them compute the evolution of the system. Finally, in the updating step construct the hyperplanes along each of the $M$ trajectories and then take the pointwise maximum of the previous subsolution with the $M$ hyperplanes thus constructed. All of our results carry over to this setting.
\end{rem}%

We now prove that at each step the algorithm produces a subsolution for the problem. To avoid complicating the notation, here we omit the dependence on $\omega$.
\begin{proposition}\label{lem:algorithm_produces_subsolutions}%
The sequence of functions $(w^{(n)})_{n\in\mathbb{N}}$ produced by the algorithm is a non-decreasing sequence of convex subsolutions for the control problem.
\end{proposition}%
\begin{proof} Recall that by assumption $w^{(0)}$ is convex. Since the pointwise maximum of convex functions is convex, by induction $w^{(n)}$ is convex for any $n$. Moreover, by construction the sequence of functions $(w^{(n)})_{n\in\mathbb{N}}$ is non-decreasing.
Finally, we prove the subsolution property by induction. By assumption $w^{(0)}$ is a subsolution. We now proceed with the inductive step and assume that $w^{(n)}$ is a subsolution. We will prove that $w^{(n+1)}$ is a subsolution by showing that
\begin{align}\label{eq:subsolution_proof_inductive_step}%
\Lambda_j^{(n+1)}(x)\eqdef\mathcal{L}_j(w^{(n+1)}(j+1,\cdot))(x)\geq w^{(n+1)}(j,x),
\end{align}%
for every $x\in\mathbb{R}^d$ and $j=1,\ldots, N$. By definition
\begin{align*}%
w^{(n+1)}(j,x) = \max\{w^{(n)}(j,\cdot), \text{Hyp}(X_j^{n+1},v_j^{n},p_j^{n})(\cdot)\}(x),
\end{align*}%
and so we will prove \eqref{eq:subsolution_proof_inductive_step} for each of the terms in the maximum. Since the sequence of functions $(w^{(n)})_{n\in\mathbb{N}}$ is non-decreasing we have $w^{(n+1)}(j+1,x)\geq w^{(n)}(j+1,x)$. By monotonicity of the Bellman operator this implies 
\begin{align*}
\Lambda_j^{(n+1)}(x)\geq \mathcal{L}_j (w^{(n)}(j\!+\!1,\cdot))(x).
\end{align*}%
Since $w^{(n)}$ is a subsolution by assumption, $\mathcal{L}_j (w^{(n)}(j\!+\!1,\cdot))(x)\geq w^{(n)}(j,x)$ and so $\Lambda_j^{(n+1)}(x)\geq w^{(n)}(j,x)$. Moreover, by convexity of $x\mapsto\Lambda_j^{(n+1)}(x)$,
\begin{align*}
\Lambda_j^{(n+1)}(x)\geq \text{Hyp}(X_j^{n+1},v_j^{n},p_j^{n})(x)
\end{align*}%
for $x\in\mathbb R^n$. Indeed, recall that  $v_j^n = \Lambda_j^{(n+1)}(X_j^{n+1})$, and $p_j^n = D^-\Lambda_j^{(n+1)}(X_j^{n+1})$.
\end{proof}
In the next sections we will prove that under some additional assumptions the sequence $(w^{(n)}(0,\cdot))_{n\in\mathbb{N}}$ converges from below to $V(0,x)$ for almost all $x\in\supp(\nu)$. Moreover, in the deterministic case, we will prove that the cost associated with the controls generated by the procedure starting in $\nu = \delta_{x_0}$ converges to $V(0,x_0)$, and the controls converge to the optimal ones.

\section{Convergence for deterministic systems} \label{sec:convergence_deterministic_systems}

For the sake of exposition, in this section we will assume that $\nu = \delta_{x_0}$ for some $x_0\in\mathcal X$, so that $X_0^n = x_0$ for all $n\in\mathbb N$. We will carry out the proof with an arbitrary starting distribution $\nu$ in the next section.

\new{
Denote by $\mathcal{X}_j(x_0)$ the set of points which can be reached through the dynamics $\Psi$ at time $j$ starting from $x_0$. Set
\[
  \mathcal{X}(x_0)\eqdef \bigcup_{j\in \{0,\ldots, N\}} \mathcal{X}_j(x_0).
\]
Notice that \hypref{hyp:system_function_continuous} and \hypref{hyp:control_space_compact} imply that $\mathcal{X}(x_0)$ is relatively compact.

\begin{proposition}\label{prop:convergence_deterministic_dynamics}
Assume \emph{\hyprefall}. Then for all $j\in\{0,\ldots,N\}$, the sequence of functions $(w^{(n)}(j,\cdot))_{n\in\mathbb{N}}$ converges uniformly on compacts to a convex function $w(j,\cdot)$. In particular, convergence is uniform on $\mathcal{X}(x_0)$. Moreover, 
\begin{equation*}%
w(0,x_0) = V(0,x_0).
\end{equation*}%
\end{proposition}
}
\begin{proof}
For every $j\in\{0,\ldots, N\}$, the sequence $( w^{(n)}(j,\cdot))_{n\in\mathbb{N}}$ is non-decreasing and bounded from above by $V(j,\cdot)$ and thus admits a pointwise limit. \new{Set
\begin{equation*}%
w(j,x)\eqdef \lim_{n\rightarrow+\infty}w^{(n)}(j,x) = \sup_{n\in\mathbb{N}}w^{(n)}(j,x),\quad j\in \{0,\ldots,N\},\; x \in \mathcal{X}. 
\end{equation*}%
Thus, for every $j\in \{0,\ldots,N\}$, $w(j,\cdot)$ is the finite-valued pointwise limit of a sequence of convex functions. This implies that $w(j,\cdot)$ is itself convex and that convergence is uniform on compact subsets of $\mathcal{X} = \mathbb{R}^{d}$; see Theorem~B.3.1.4 in \cite[p.\,105]{HUL2001}. In particular, since $\mathcal{X}(x_0)$ is relatively compact, $(w^{(n)}(j,\cdot))_{n\in\mathbb{N}}$ converges to $w(j,\cdot)$ uniformly on $\mathcal{X}(x_0)$. 
}

We are left to prove convergence at the starting point $x_0$. We have shown that $w(0,x_0)\leq V(0,x_0)$, and so we are left to prove the opposite inequality. Fix $\varepsilon >0$. Since $w^{(n)}(j,\cdot)$ converges uniformly on $\mathcal{X}(x_0)$, there exists $\bar n(\varepsilon)\in\mathbb{N}$ such that for all $n\geq \bar n(\varepsilon)$, all $j\in \{0,\ldots,N\}$, all $x\in\mathcal{X}(x_0)$, we have
\begin{equation*}%
  w(j,x)-w^{(n)}(j,x)\leq\varepsilon.
\end{equation*}%
By construction 
\begin{equation}\label{eq:update_along_trajectories}%
w^{(n+1)}(j,X_j^n) = \mathcal{L}_j(w^{(n+1)}(j+1,\cdot))(X_j^n),
\end{equation}%
so that for every $j$ we have
\begin{align*}%
w(j,X_j^n)\geq  w^{(n+1)}(j,X_j^n) &= \mathcal{L}_j(w^{(n+1)}(j+1,\cdot))(X_j^n) \\
& \geq \mathcal{L}_j(w^{(n)}(j+1,\cdot))(X_j^n) \\
& = \inf_{\gamma\in\Gamma}\{f(j,X_j^n,\gamma) + w^{(n)}(j+1,\Psi(j,X_j^n,\gamma))\},
\end{align*}%
where, in the third inequality, we used the monotonicity of the Bellman operator. The choice of controls in the trajectory simulation step minimizes the right-hand side of the equation. When $n\geq\bar{n}(\varepsilon)$ we get
\begin{align*}%
w(j,X_j^n) & \geq f(j,X_j^n,\gamma_j^n)+w^{(n)}(j+1,\Psi(j,X_j^n,\gamma_j^n)) \\ 
&\geq f(j,X_j^n,\gamma_j^n)+w(j+1,\Psi(j,X_j^n,\gamma_j^n)) - \varepsilon.
\end{align*}
Iterating the inequality gives
\begin{align}\label{eq:subsol_converges_to_value_function}
w(0,x_0) &\geq \sum_{j=0}^{N-1}f(j,X_j^n,\gamma_j^n)+w(N,X_j^n)-N\varepsilon \notag\\
&\geq \sum_{j=0}^{N-1}f(j,X_j^n,\gamma_j^n)+F(X_j^N)-N\varepsilon \notag\\
& =  J(0,x_0,(\gamma_j^n)) - N\varepsilon \geq  V(0,x_0) -N\varepsilon.
\end{align}%
Letting $\varepsilon\rightarrow0$ concludes the proof.
\end{proof}

\new{
Next, we show that the trajectories generated by the algorithm converge along subsequences and that any limiting trajectory is optimal. In fact, since $\mathcal X(x_0)$ is relatively compact and the control space $\Gamma$ is compact, we can extract from $(X^{n}_j, \gamma^{n}_{j})_{j}$ a converging subsequence. In other words, there exists a sequence $(n_k)_{k\in\mathbb N}$ such that $(X^{n_k}_j, \gamma^{n_k}_{j})_{j}\to(X_j,\gamma_j)_j\in(\cl(\mathcal X)(x_0)\times\Gamma)^N$ as $k\to\infty$. This convergence, together with Assumption~\hypref{hyp:cost_coefficients}, implies that 
\begin{align*}%
J(0,x_{0},(\gamma^{n_k}_{j})_j ) \stackrel{k\to\infty}{\longrightarrow}  J(0,x_{0},(\gamma_{j})_j ).
\end{align*}%
The computations in the proof of Proposition \ref{prop:convergence_deterministic_dynamics} show that for every $\varepsilon>0$, and $n\geq \bar n(\varepsilon)$,
\begin{align}\label{eq:costs_associated_to_generated_trajectories_upper_lower_bounds}%
    V(0,x_0) \leq J(0,x_{0},(\gamma^{n}_{j}) ) \leq V(0,x_0) + N\varepsilon.
\end{align}%
By taking the limit in \eqref{eq:costs_associated_to_generated_trajectories_upper_lower_bounds} first as $n=n_k\to\infty$ and then $\varepsilon\to0$, we get 
\begin{align*}%
J(0,x_{0},(\gamma_{j})) = V(0,x_{0}).
\end{align*}%
We summarize the computations above in the following result. 
\begin{theorem}%
Assume \emph{\hyprefall}. For any subsequence $(n_k)_{k\in\mathbb N}$, there exists a sub-subsequence $(n_{k_i})_{i\in\mathbb N}$ such that the control actions and state trajectories produced by the algorithm converge along the sub-subsequence to the optimal ones. Moreover, the subsolutions converge along the same sub-subsequence to the value function at each point of the optimal trajectory.
\end{theorem}%
\begin{proof}
We are left to prove the last claim. Repeating the computations in \eqref{eq:subsol_converges_to_value_function} for a generic time $m\in\{0,\ldots,N\}$ gives
\begin{align*}%
w(m,X_{m}^n) &\geq \sum_{j=m}^{N-1}f(j,X_j^n,\gamma_j^n)+F(X_N^n) - N\varepsilon = J(m,X^{n}_m,(\gamma_j^n))-N\varepsilon \\
&\geq V(m,X_m^n)-N\varepsilon.
\end{align*}%
Since $(X_j^{n_{k_i}})_j\rightarrow (X _j)_j$ as $i\to\infty$, we get
\begin{align*}%
w(m, X _m)=\lim_{i\rightarrow\infty}w(m,X_m^{n_{k_i}}) &= \liminf_{i\rightarrow\infty}w(m,X_m^{n_{k_i}})\\
&\geq \liminf_{i\rightarrow\infty}V(m,X_m^{n_{k_i}})-N\varepsilon = V(m, X_m) - N\varepsilon.
\end{align*}%
In the first equality above we used the continuity of $x\mapsto w(m,x)$ and in the last inequality the continuity of the value function. Letting $\varepsilon\rightarrow0$, we get $w(m, X_m) \geq V(m, X_m)$. The reverse inequality always holds and hence $w(m, X_m) = V(m, X_m)$.
\end{proof}
}


\section{Convergence for stochastic systems} \label{sec:convergence_stochastic_systems}

\new{ As in Section~\ref{sec:algorithm}, let $(\Omega, \mathcal{F},\mathbf{P})$ be a probability space carrying a family $(\xi_j^{(n)})_{j\in\mathbb N, n\in \mathbb{N}_0}$ of independent noise variables and a sequence $(\zeta^{(n)})_{n\in\mathbb N_0}$ of i.i.d.\ random variables with common law $\nu$, independent from $(\xi_j^{(n)})_{j\in\mathbb N, n\in \mathbb{N}_0}$. Define filtrations $(\mathcal{F}^{n}_j)$, $n\in \mathbb{N}_0$, by setting
\[
  \mathcal{F}^{n}_j \doteq \sigma\left( \zeta^{(k)}, \xi^{(k-1)}_{i}, \xi^{(n)}_{l} : k\in \{0,\ldots,n\},\, i\in \mathbb{N},\, l\in \{1,\ldots,j\} \right),\quad j\in\mathbb N_0,
\]
where $\sigma(\ldots)$ indicates the $\sigma$-algebra generated by the random elements indicated inside the brackets, $\{1,\ldots,j\} = \emptyset$ if $j = 0$, and $\xi^{(-1)}_{i}$ is a void element. Further set
\[
  \mathcal{F}^{n}_{\infty} \doteq \sigma\left( \zeta^{(k)}, \xi^{(k)}_{i} : k\in \{0,\ldots,n\},\, i\in \mathbb{N} \right),\quad n\in\mathbb N_0.
\]
Notice that $\mathcal{F}^{k}_i \subseteq \mathcal{F}^{n}_j \subseteq \mathcal{F}^{n}_{\infty}$ for all $i, j \in \mathbb{N}_{0}$ whenever $k < n$. In particular, $\mathcal{F}^{n}_0$ contains $\mathcal{F}^{n-1}_{\infty}$, which in turn contains the $\sigma$-algebra generated by all preceding filtrations $(\mathcal{F}^{k}_j)$, $k\in \{0,\ldots,n-1\}$. Also notice that $\zeta^{(n)}$ is $\mathcal{F}^{n}_0$-measurable and independent of $\mathcal{F}^{n-1}_{\infty}$, while $w^{(n)}(0,\cdot)$, the subsolution at time zero constructed at the $n$-th iteration of the algorithm (seen as a random element with values in the space of continuous functions on $\mathcal{X}$), is $\mathcal{F}^{n-1}_{\infty}$-measurable, hence independent of $\zeta^{(n)}$.

We will prove that the subsolutions generated by the algorithm converge to the value function at the initial time under general assumptions on the noise distribution and on the initial distribution. The proof proceeds similarly to the deterministic case, however there is a crucial technical difference: when the distribution $\nu$ of the initial states has unbounded support, then $\mathcal X(x_0)$, the set of reachables states, is not necessarily relatively compact. The following condition turns out to be sufficient to carry out the proof.
\begin{hypenv}%
\item\label{hyp:integrability_condition}
Given any sequence  $((\gamma_j^n))_{n\in \mathbb{N}_{0}}$ of stochastic control policies such that, for every $n \in \mathbb{N}_{0}$, $(\gamma_j^n)$ is adapted to the filtration $(\mathcal{F}^{n}_j)$, there exists $p>1$ such that
\[%
  \sup_{n\in\mathbb N_0} \Mean\Big[\Big(\sum_{j=0}^{N-1} f(j,X_j^{n},\gamma_j^{n})+ F(X_N^{n})\Big)^{p}\Big] + \max_{j\in \{0,\ldots,N\}} \Mean[\vert X_j^n\vert^p]<\infty,
\]
where $X^{n}_{j}$, $j\in \{0,\ldots,N\}$, is the controlled state sequence according to \eqref{eq:state_recursion_weak} with $\gamma_j = \gamma^n_j$, $\xi_j = \xi^{(n)}_j$ starting at time $j_0 = 0$ from $\zeta^{(n)}$ (instead of $x_0$). 
\end{hypenv}%
}

Assumption \hypref{hyp:integrability_condition} implies that, given any sequence $((\gamma_j^n))$ of stochastic control policies, the family of random variables $\{X_j^n\}_{j\in\{0,\ldots,N\}, n\in \mathbb{N}_{0}}$ is tight. 

\new{
\begin{rem}
Let $p > 1$, and suppose the cost coefficients $f$, $F$ have sub-polynomial growth of order $p$ in the state variable, that is, for some $K > 0$,
\begin{align*}%
\sup_{\gamma \in \Gamma} \vert f(j,x,\gamma)\vert \vee \vert F(x) \vert\leq K\left(1 + \vert x \vert^p \right) \text{ for all } x\in \mathcal{X},\; j\in \{0,\ldots,T-1\}. 
\end{align*}%
If the initial distribution $\nu$ has finite $p$-th absolute moment, that is, $\int_{\mathcal{X}} |x|^{p} \nu(\mathrm d x) < \infty$, and the noise distribution $\mu$ has bounded support, then Assumption \hypref{hyp:integrability_condition} holds for the control problems from Examples \ref{ex:linear_convex_models} and \ref{ex:monotone_convex_models}. In the case of Example~\ref{ex:linear_convex_models} (linear-convex models), $\mu$ may have unbounded support provided it has finite $p$-th absolute moment as well. In the case of Example~\ref{ex:monotone_convex_models} (monotone-convex models), an additional growth condition on the dynamics would be required when $\mu$ has unbounded support.
\end{rem}

We are now in a position to prove our main convergence result.
\begin{theorem}%
Assume \hypref{hyp:system_function_continuous}--\hypref{hyp:integrability_condition}. Let $\supp(\nu)$ denote the support of the common initial state distribution $\nu$. Then for $\Prb$-almost every $\omega\in\Omega$,
\begin{equation*}%
  \lim_{n\to\infty} w^{(n)}_{\omega}(0,x) = V(0,x) \quad \text{for every } x\in \supp(\nu),
\end{equation*}%
uniformly on compact subsets of $\supp(\nu)$.
\end{theorem}%
}
\begin{proof}
Proceeding as in the proof of Proposition~\ref{prop:convergence_deterministic_dynamics}, we see that for every $\omega\in\Omega$, $w^{(n)}_{\omega}(j,\cdot)$ converges from below to some function $w_{\omega}(j,\cdot)$ uniformly on compact subsets of $\mathcal{X} = \mathbb{R}^d$, and that $w_{\omega}(j,\cdot)$ is convex. 

Fix $\varepsilon >0$. We will show that there exists $\bar{n} = \bar{n}(\varepsilon) \in\mathbb N$ such that the event
\begin{equation*}%
  M({\varepsilon}) \eqdef \left\{\omega\in\Omega:\,w_{\omega}(j,X_j^{\bar n}(\omega)) - w_{\omega}^{(\bar n)}(j,X_j^{\bar n}(\omega))\leq\varepsilon \text{ for all } j\in\{0,\ldots,N\} \right\}
\end{equation*}%
occurs with probability greater than or equal to $1-\varepsilon$. Assumption \hypref{hyp:integrability_condition} implies that the family $(X_j^n)_{j\in\{0,\ldots,N\},n\in\mathbb N_0}$ is tight. Hence there exists a large compact set $K({\varepsilon})$ such that for every $n \in \mathbb{N}_0$,
\begin{equation*}%
  \Prb (M_{1,n}^{\varepsilon}) \geq 1-\frac{\varepsilon}{2} \quad \text{where } M_{1,n}^{\varepsilon} \doteq \left\{ \omega \in \Omega :  X_j^n(\omega) \in K({\varepsilon}) \text{ for all } j\in\{0,\ldots,N\} \right\}.
\end{equation*}%
Since $w^{(n)}_{\omega}(j,\cdot)$ converges uniformly on $K(\varepsilon)$ to $w_{\omega}(j,\cdot)$ for every $\omega\in \Omega$, there exists $n_{\varepsilon}(\omega) \in \mathbb N$ such that $w_{\omega}(j,x) - w_{\omega}^{(n)}(j,x)\leq\varepsilon$ for all $n\geq n_{\varepsilon}(\omega)$, $j\in\{0,\ldots,N\}$, $x\in K(\varepsilon)$. Since
\[
  \Prb\left( \{\omega\in \Omega : n_{\varepsilon}(\omega)\leq L\} \right)\nearrow 1 \text{ as } L\to \infty,
\] 
we can choose $\bar{n} = \bar{n}(\varepsilon)$ such that
\begin{equation*}%
\Prb(M_{2}^{\varepsilon})\geq 1-\frac{\varepsilon}{2}\quad \text{where } M_{2}^{\varepsilon} \eqdef \left\{ \omega\in \Omega : n_{\varepsilon}(\omega)\leq \bar{n} \right\}.
\end{equation*}%
For all $\omega\in M_{1, \bar n}^{\varepsilon} \cap M_2^{\varepsilon}$ it holds that 
\begin{align*}%
w_{\omega}(j,X_j^{\bar n}(\omega)) - w_{\omega}^{(\bar n)}(j,X_j^{\bar n}(\omega))\leq\varepsilon \quad \text{for all } j\in\{0,\ldots,N\}.
\end{align*}%
In other words, we have $M_{1, \bar n}^{\varepsilon} \cap M_2^{\varepsilon} \subseteq M(\varepsilon)$. It follows that
\begin{align*}
  \Prb\left(M(\varepsilon)\right)\geq \Prb \left(M_{1, \bar n}^{\varepsilon} \cap M_2^{\varepsilon} \right) \geq \Prb\left(M_{1, \bar n}^{\varepsilon}\right) - \Prb\left((M_2^{\varepsilon})^{c}\right) \geq 1-\frac{\varepsilon}{2} - \frac{\varepsilon}{2} = 1 - \varepsilon.
\end{align*}%
For any $\omega\in\Omega$, $j\in\{0,\ldots,N-1\}$, we estimate
\begin{align*}
w_{\omega}(j,X_j^{\bar n}(\omega)) &\geq w_{\omega}^{(\bar n+1)}(j,X_j^{\bar n}(\omega)) = \mathcal L _j(w_{\omega}^{(\bar n+1)}(j+1,\cdot))(X_j^{\bar n}(\omega))\\
&\geq \mathcal L _j(w_{\omega}^{(\bar n)}(j+1,\cdot))(X_j^{\bar n}(\omega))\\
&= \inf_{\gamma\in\Gamma}\Big\{f(j,X_j^{\bar n}(\omega),\gamma) + \int_{\mathcal{Y}} w_{\omega}^{(\bar n)}\bigl(j+1,\Psi(j,X_j^{\bar n}(\omega),\gamma,y)\bigr) \mu(\mathrm d y)\Big\}\\
&= f(j,X_j^{\bar n}(\omega),\gamma_j^{\bar n}(\omega)) + \int_{\mathcal{Y}} w_{\omega}^{(\bar n)}\bigl(j+1,\Psi(j,X_j^{\bar n}(\omega),\gamma_j^{\bar n}(\omega),y)\bigr) \mu(\mathrm d y),
\end{align*}
where $\mu$ is the common law of the noise variables $(\xi_j)_{j\in\mathbb N}$. For any $\omega\in M(\varepsilon)$ we can further bound
\begin{multline}\label{eq:recursive_inequality_along_trajectories}%
  w_{\omega}(j,X_j^{\bar n}(\omega)) \geq  f(j,X_j^{\bar n}(\omega),\gamma_j^{\bar n}(\omega)) + \int_{\mathcal{Y}} w_{\omega}^{(\bar n)}\bigl(j+1,\Psi(j,X_j^{\bar n}(\omega),\gamma_j^{\bar n},y)\bigr) \mu(\mathrm d y) \\
  + w_{\omega}(j+1,X_{j+1}^{\bar n}(\omega)) - w_{\omega}^{(\bar n)}(j+1,X_{j+1}^{\bar n}(\omega)) - \varepsilon.
\end{multline}%
By iterating \eqref{eq:recursive_inequality_along_trajectories} we get, for any $\omega\in M(\varepsilon)$,
\begin{multline} \label{eq:iterated_recursive_inequality_along_trajectories}
  w_{\omega}(0,X_0^{\bar n}(\omega)) \geq \left(\sum_{j=0}^{N-1} f(j,X_j^{\bar n},\gamma_j^{\bar n}(\omega)) + F(X_N^{\bar n}(\omega))\right) - N\varepsilon  \\
  + \sum_{j=0}^{N-1} \Bigl( \int_{\mathcal{Y}} w_{\omega}^{(\bar n)}\bigl(j+1,\Psi(j,X_j^{\bar n}(\omega),\gamma_j^{\bar n},y) \bigr)\mu (\mathrm d y)- w_{\omega}^{(\bar n)}(j+1,X_{j+1}^{\bar n}(\omega)) \Bigr).
\end{multline}%
\new{ Now, we multiply both sides in \eqref{eq:iterated_recursive_inequality_along_trajectories} by the indicator function of $M(\varepsilon)$ and take expectation. We analyze the resulting first and third term separately. For the first term, write
\begin{multline} \label{eq:expected_value_iterated_recursive_inequality_along_trajectories}%
  \Mean \Bigl[\Big(\sum_{j=0}^{N-1} f(j,X_j^{\bar n},\gamma_j^{\bar n}) + F(X_N^{\bar n})\Big)\cdot \mathbf 1 _{M(\varepsilon)} \Bigr] \\
  = \Mean \Bigl[\sum_{j=0}^{N-1} f(j,X_j^{\bar n},\gamma_j^{\bar n})+ F(X_N^{\bar n})\Bigr] - \Mean\Bigl[ \Big(\sum_{j=0}^{N-1} f(j,X_j^{\bar n},\gamma_j^{\bar n})+ F(X_N^{\bar n})\Big)\cdot \mathbf 1 _{M(\varepsilon)^c}\Bigr].
\end{multline}%
We lower bound the second term in \eqref{eq:expected_value_iterated_recursive_inequality_along_trajectories} using H\"older's inequality as
\begin{multline*}%
  -\Mean\Bigl[\Big(\sum_{j=0}^{N-1} f(j,X_j^{\bar n},\gamma_j^{\bar n})+ F(X_N^{\bar n})\Big)\cdot \mathbf 1 _{M(\varepsilon)^c}\Bigr] \\
  \geq -\Mean\Bigl[\Big(\sum_{j=0}^{N-1} f(j,X_j^{\bar n},\gamma_j^{\bar n})+ F(X_N^{\bar n})\Big)^{p}\Bigr]^{1/p}\cdot \Prb\left(M(\varepsilon)^c\right)^{(p-1)/p},
\end{multline*}%
where $p\in (1,\infty)$ is given by Assumption~\hypref{hyp:integrability_condition}. By \hypref{hyp:integrability_condition}, the term on the left-hand side is bounded, while $\Prb(M(\varepsilon)^c) \leq \varepsilon$.

Next, we discuss the third term on the right-hand side of \eqref{eq:iterated_recursive_inequality_along_trajectories}. By construction,
\[
  w_{\omega}^{(\bar n)}\bigl(j+1,X_{j+1}^{\bar n}(\omega) \bigr) = w_{\omega}^{(\bar n)} \bigl(j+1,\Psi(j,X_j^{\bar n}(\omega),\gamma_j^{\bar n}(\omega),\xi^{(\bar n)}_{j+1}(\omega)) \bigr) \text{ for all } \omega \in \Omega,
\]
where $\xi^{(\bar n)}_{j+1}$ has distribution $\mu$ and is independent of $w^{(\bar n)}(j+1,\Psi(j,X_j^{\bar n},\gamma_j^{\bar n},\cdot))$, seen as a function-valued random element. It follows that
\[
  \Mean\Bigl[ \int_{\mathcal{Y}} w^{(\bar n)}\bigl(j+1,\Psi(j,X_j^{\bar n},\gamma_j^{\bar n},y) \bigr) \mu (\mathrm d y)- w^{(\bar n)}\bigl(j+1,X_{j+1}^{\bar n}\bigr) \Bigr] = 0,
\]
hence, since $0 \leq w^{(\bar n)}_{\omega}(j+1,\cdot) \leq V(j+1,\cdot)$ for all $\omega \in \Omega$,
\begin{align*}%
  \Mean\Bigl[ &\Big(\int_{\mathcal{Y}} w^{(\bar n)}\bigl(j+1,\Psi(j,X_j^{\bar n},\gamma_j^{\bar n},y) \bigr) \mu (\mathrm d y)- w^{(\bar n)}\bigl(j+1,X_{j+1}^{\bar n}\bigr) \Big)\cdot \mathbf 1 _{M(\varepsilon)}\Bigr] \\
  &= -\Mean\Bigl[ \Big(\int_{\mathcal{Y}} w^{(\bar n)}\bigl(j+1,\Psi(j,X_j^{\bar n}, \gamma_j^{\bar n},y) \bigr) \mu (\mathrm d y) - w^{(\bar n)}(j+1,X_{j+1}^{\bar n}) \Big)\cdot \mathbf 1 _{M(\varepsilon)^c}\Bigr] \\
  &\geq -\Mean\Big[ \int_{\mathcal{Y}} V\bigl(j+1, \Psi(j,X_j^{\bar n}, \gamma_j^{\bar n},y)\bigr) \mu (\mathrm d y)\cdot \mathbf 1 _{M(\varepsilon)^c}\Big] \\
  &\geq -\Mean\Big[ \Big(\int_{\mathcal{Y}} V\bigl(j+1,\Psi(j,X_j^{\bar n},\gamma_j^{\bar n},y)\bigr) \mu (\mathrm d y)\Big)^p\Big]^{1/p} \cdot \Prb(M(\varepsilon)^c)^{(p-1)/p},
\end{align*}%
where the last inequality above is again a consequence of H\"older's inequality with $p\in(1,\infty)$ as given by Assumption \hypref{hyp:integrability_condition}. By Jensen's inequality, the definition of the dynamics, and the definition of the value function, we have
\begin{align*}%
  \Mean\Big[ \Big(\int_{\mathcal{Y}} V\bigl(j+1,\Psi(j,X_j^{\bar n},\gamma_j^{\bar n},y)\bigr) \mu (\mathrm d y)\Big)^p\Big] &\leq \Mean\Big[ V\bigl(j+1,\Psi(j,X_j^{\bar n},\gamma_j^{\bar n},\xi_{j+1}^{\bar n})\bigr)^p\Big] \\
%
%
  &\leq \Big[\Big(\sum_{i=j+1}^{N-1} f\bigl(i,X_i^{\bar n},\gamma_i^{\bar n}\bigr) + F\bigl(X_N^{\bar n}\bigr) \Big)^p\Big],
\end{align*}%
and the right-hand side above is finite by Assumption \hypref{hyp:integrability_condition}. To summarize, we have shown that, for some finite constant $C > 0$ not depending on $\varepsilon$,
\begin{align*}%
  \Mean \left[ w\bigl(0,X_{0}^{\bar n} \bigr)\cdot \mathbf{1}_{M(\varepsilon)} \right] &\geq \Mean\Big[  \sum_{j=0}^{N-1} f\bigl(j,X_j^{\bar n},\gamma_j^{\bar n}\bigr) + F\bigl(X_N^{\bar n}\bigr)  \Big] -N\varepsilon - C\varepsilon \\
  &= \int_{\mathcal{X}} J\bigl(0,x,(\gamma_j^{\bar n, x}) \bigr)\nu(\mathrm d x) \;- (N+C)\varepsilon,
\end{align*}
where $(\gamma_j^{\bar n, x})$ denotes the stochastic control policy (in the sense of Definition~\ref{def:stochastic_control_policy}) derived from $(\gamma_j^{\bar n})$ by replacing the underlying probability measure $\Prb$ with $\Prb( \,\cdot\, | X_{0}^{\bar n} = x)$, that is, with the probability kernel corresponding to (a version of) the regular conditional distribution of the identity map on $\Omega$ given $X_{0}^{\bar n}$ evaluated in $x \in \mathcal{X}$. Also recall that $X_{0}^{\bar n}$ has distribution $\nu$. It follows that
\begin{align} \label{eq:wzero_estimate}
  \Mean \left[ w\bigl(0,X_{0}^{\bar n} \bigr)\cdot \mathbf{1}_{M(\varepsilon)} \right] &\geq  \int_{\mathcal{X}} J\bigl(0,x,(\gamma_j^{\bar n, x}) \bigr) \nu(\mathrm d x) \;- (N+C)\varepsilon \\
  &\geq \int_{\mathcal{X}} \inf_{(\gamma_j)\in\mathcal R} J\bigl(0,x,(\gamma_j) \bigr) \nu(\mathrm d x) \;- (N+C)\varepsilon \notag \\ 
&=  \int_{\mathcal{X}} V(0,x) \nu(\mathrm d x) \;- (N+C)\varepsilon, \notag
\end{align}%
where the equality in the last line above is due to Lemma \ref{lem:value_functions_strong_weak_equal}. Recall that $w^{(\bar n)}(0,\cdot)$ is $\mathcal{F}_{\infty}^{\bar n -1}$-measurable, hence independent of $X^{\bar n}_{0} = \zeta^{(\bar n)}$. Using this, the definition of $M(\epsilon)$, the fact that $w_{\omega}(0,\cdot) \geq w^{(\bar n)}_{\omega}(0,\cdot) \geq 0$ for all $\omega \in \Omega$, as well as \eqref{eq:wzero_estimate}, we find that
\begin{align*}
  \Mean \left[ \int_{\mathcal{X}} w(0,x) \nu(\mathrm d x) \right] &\geq \Mean \left[ \int_{\mathcal{X}} w^{(\bar n)}(0,x) \nu(\mathrm d x) \right] = \Mean \left[ w^{(\bar n)}\bigl(0,X_{0}^{\bar n} \bigr) \right]\\
  &\geq \Mean \left[ w^{(\bar n)}\bigl(0,X_{0}^{\bar n} \bigr) \cdot \mathbf{1}_{M(\varepsilon)} \right] \\
  &\geq \Mean \left[ w\bigl(0,X_{0}^{\bar n} \bigr)\cdot \mathbf{1}_{M(\varepsilon)} \right] - \varepsilon \\
&\geq  \int_{\mathcal{X}} V(0,x) \nu(\mathrm d x) \;- (N+1 + C)\varepsilon.
\end{align*}
Since $V(0,\cdot) \geq w_{\omega}(0,\cdot)$ for all $\omega \in \Omega$, it follows that
\begin{multline*}
  \Mean \left[ \int_{\mathcal{X}} \bigl| V(0,x) - w(0,x) \bigr|\, \nu(\mathrm d x) \right] = \int_{\mathcal{X}} V(0,x) \nu(\mathrm d x) - \Mean \left[ \int_{\mathcal{X}} w(0,x) \nu(\mathrm d x) \right] \\
  \leq (N+1 + C)\varepsilon,
\end{multline*}
hence, by letting $\varepsilon \to 0$,
\begin{equation} \label{eq:V_equals_w_Pnuas}
  \Mean \left[ \int_{\mathcal{X}} \bigl| V(0,x) - w(0,x) \bigr|\, \nu(\mathrm d x) \right] = 0.
\end{equation}

From \eqref{eq:V_equals_w_Pnuas} we are going to deduce by contradiction that for $\Prb$-almost every $\omega \in \Omega$,
\begin{equation} \label{eq:convergence_claim}
  V(0,x) = w_{\omega}(0,x) \text{ for all } x \in \supp(\nu).
\end{equation}
Suppose on the contrary that there are $\tilde x\in\mathbb\supp(\nu)$ and $A \in \mathcal{F}$ such that $\Prb(A) > 0$ and $w_{\omega}(0,\tilde x) \neq V(0,\tilde x)$ for every $\omega \in A$. This would imply, since $V(0,\cdot) \geq w_{\omega}(0,\cdot)$ and $V(0,\cdot)$, $w_{\omega}(0,\cdot)$ are continuous (being convex and finite-valued) for every $\omega \in \Omega$, that
\[
  A \subseteq \bigcup_{n,m\in \mathbb{N}} A_{n,m} \text{ with } A_{n,m} \doteq \left\{ \omega \in \Omega : V(0,x) - w_{\omega}(0,x) \geq \frac{1}{m} \text{ for all } x\in B_{\frac{1}{n}}(\tilde x) \right\},
\]
where $B_{r}(\tilde x)$ denotes the open ball in $\mathcal{X}$ of radius $r$ centered in $\tilde x$. If $\Prb(A) > 0$, there would hence exist $\tilde n, \tilde m \in \mathbb{N}$ such that $\Prb(A_{\tilde n, \tilde m}) > 0$. But then, since $\nu(B_{\frac{1}{\tilde n}}(\tilde x)) > 0$ as $\tilde x\in\mathbb\supp(\nu)$ by definition of the support of a probability measure, we would have
\[
  \Mean \left[ \int_{\mathcal{X}} \bigl| V(0,x) - w(0,x) \bigr|\, \nu(\mathrm d x) \right] \geq \frac{1}{\tilde m}\cdot \Prb(A_{\tilde n, \tilde m})\cdot \nu(B_{\frac{1}{\tilde n}}(\tilde x)) > 0,
\]
which contradicts \eqref{eq:V_equals_w_Pnuas}. We conclude that \eqref{eq:convergence_claim} holds.
}

\end{proof}

\section{Numerical experiments}\label{sec:numerical_experiments}

For the numerical tests, we will use linear-convex models 
derived by time discretization from continuous time problems of the form:
\begin{align*}
  &\text{minimize } \Mean\left[ \int_{0}^{T} \left( c|u(t)|^{2} + \bar{f}(X(t)) \right)dt + F(X(T) )\right] \\
  &\text{subject to } dX(t) = \left(A X(t) + B u(t)\right)dt + C dW(t),\quad X(0) = x_{0}, \\
  &\text{over } \bar{B}(0,r)\text{-valued adapted processes }u,
\end{align*}
where $c \geq 0$, $\bar{B}(0,r)$ is the closed ball of radius $r$ centered at the origin of $\mathbb{R}^{d_{2}}$, $W$ is a $d_{1}$-dimensional standard Wiener process, $A$, $B$, $C$ are matrices of dimensions $d\times d$, $d\times d_2$, and $d\times d_1$, respectively, and $\bar{f}$, $F$ are convex functions. If the matrix $C$ is equal to zero, then the above stochastic optimal control problem reduces to the deterministic problem:
\begin{align*}
  &\text{minimize } \int_{0}^{T} \left( c|u(t)|^{2} + \bar{f}(x(t)) \right)dt + F(x(T)) \\
  &\text{subject to } dx(t) = \left(A x(t) + B u(t)\right)dt,\quad x(0) = x_{0}, \\
  &\text{over } \bar{B}(0,r)\text{-valued measurable functions }u.
\end{align*}

Let $h > 0$ be such that $N\doteq T / h$ is an integer. A standard Euler-type time discretization with step size $h$ yields the stochastic discrete time problem 
\begin{align*}
  &\text{minimize } J(0,x_{0},(\gamma_{j})) \doteq \Mean\left[ \sum_{j=0}^{N-1} f(j,X_{j},\gamma_{j}) + F(X_{N}) \right] \\
  &\text{subject to } X_{j+1}\doteq \Psi(j,X_{j},\gamma_{j},\xi_{j+1}),\quad X_0 = x_{0}, \\
  &\text{over adapted } \Gamma\text{-valued control sequences }(\gamma_{j}),
\end{align*}
where $\Gamma \doteq \bar{B}(0,r)$, the terminal costs $F$ are as above,
\begin{align*}
  f(j,x,\gamma) &\doteq \left( c|\gamma|^{2} + \bar{f}(x) \right)h, \\
  \Psi(j,x,\gamma,y) &\doteq x + \left(Ax + B\gamma\right)h + \sqrt{h}C y,  
\end{align*}
and $\xi_{1}, \xi_{2}, \ldots$ are i.i.d.\ $\mathbb{R}^{d_1}$-valued random variables with common distribution $\mu$. As noise distribution $\mu$, one could choose the $d_{1}$-variate standard normal distribution. For ease of implementation, we will take $\mu \doteq \otimes^{d_{1}} \mathrm{Rad}(1/2)$, the $d_{1}$-fold product of the Rademacher distribution concentrated in $1$ and $-1$ with equal probability, in accordance with the simple random walk approximation of the standard Wiener process. If the matrix $C$ equals zero, then we have a deterministic optimal control problem.

\new{We present five numerical examples of the above type, two deterministic and three stochastic.} The time horizon $T$ will be fixed at $T = 2$. The time step $h$ will be less than or equal to $1/100$ so that at each iteration of our procedure there will be at least $200$ steps in time. In the actual implementation, when computing optimal control actions at any step in discrete time, we will evaluate the Hamiltonian of the continuous time problem, exploiting the fact that, given any subgradient $p \in \mathbb{R}^{d}$, the argmin of the mapping
\[
  \gamma \mapsto c|\gamma|^{2} + \langle p, B\gamma \rangle \text{ over } \gamma \in \bar{B}(0,r)
\]
can be calculated analytically. This ``Hamiltonian approximation'' works well provided the time step $h$ is sufficiently small.

The cost coefficients $F$, $\bar{f}$ will be constant or (affine-)quadratic. Our example problems might thus appear to be classical linear-quadratic problems. However, due to the fact that the space of control actions is bounded, the corresponding value function will in general not be of linear-quadratic type.

The numerical tests were done using a 2015 Fujitsu W570 Linux workstation with Intel Xeon \mbox{E3-1271v3} \mbox{CPU} at 3.60GHz. We implemented the iterative procedure in \mbox{C++} using its standard template library. No parts of the procedure were parallelized, although this would certainly be possible at different points. For instance, the evaluation of functions given as pointwise maxima of hyperplanes could be carried out in parallel.

\subsection{Deterministic case}\label{sec:deterministic_numerical_experiments}

Given an initial state $x_{0}\in \mathbb{R}^{d}$, the optimal control problem is to minimize
\[
  \sum_{j=0}^{N-1} \left( c|\gamma_{j}|^{2} + \bar{f}(x_{j}) \right)h + F(x_{N}),\quad \text{where }N \doteq T/h,
\]
over $\bar{B}(0,r)$-valued control sequences $(\gamma_{j})$ subject to the dynamics
\[
  x_{j+1} = x_{j} + \left(A x_{j} + B \gamma_{j}\right)h, \quad j\in \{0,\ldots,N-1\}.
\]

\begin{example}[5-dimensional deterministic] \label{CDSysLinQuad5D}
Dimensions $d = d_2 = 5$, control actions in the closed ball $\Gamma \doteq \bar{B}(0,1)$,
\begin{align*}
  &A = 0,& &B = \Id,& &\bar{f} \equiv 0,& & F(x) \doteq 1 + \|x\|^{2}.&
\end{align*}
Table~\ref{TabCDSysLinQuad5D} summarizes data of a numerical test with time horizon $T = 2$ and initial state $x_{0}\doteq (1, -\sqrt{3}, 2, 1, -1)$ starting from the constant subsolution $w^{(0)} \equiv 0$. The first two columns there give the time discretization parameter $h$ and the cost coefficient $c$ for the quadratic control running costs, respectively. The third column shows the number of iterations of the algorithm, while the fourth column reports the value of the subsolution at initial state computed after the indicated number of iterations. The difference between this value from below and the value from above, namely the cost associated with the trajectory obtained from the feedback strategy induced by the last subsolution, is given in column five. The last column shows the computing time in seconds.

\begin{table} 
\begin{center}
  \begin{tabular}{l|c|c|c|c|c}
    $h$ & $c$ & \#iter. & value below & $\Delta_{(above - below)}$ & \new{CPU time}  \\
    \hline
    $0.01$ & $0.0$ & 20 & 2.35 & $-5.46\cdot 10^{-14}$ & 0.006s \\
    $0.01$ & $0.0$ & 50 & 2.35 & $-5.46\cdot 10^{-14}$ & 0.027s \\
    $0.0001$ & $0.0$ & 20 & 2.35 & $-2.44\cdot 10^{-12}$ & 0.472s \\
    $0.0001$ & $0.0$ & 50 & 2.35 & $-2.44\cdot 10^{-12}$ & 2.421s \\[1ex]
    $0.01$ & $0.5$ & 20 & 3.35 & $-1.38\cdot 10^{-14}$ & 0.006s \\
    $0.01$ & $0.5$ & 50 & 3.35 & $-1.38\cdot 10^{-14}$ & 0.023s \\
    $0.0001$ & $0.5$ & 20 & 3.35 & $-2.08\cdot 10^{-13}$ & 0.464s \\
    $0.0001$ & $0.5$ & 50 & 3.35 & $-2.08\cdot 10^{-13}$ & 2.424s \\[1ex]
    $0.01$ & $1.5$ & 20 & 5.29 & $1.78\cdot 10^{-4}$ & 0.008s \\
    $0.01$ & $1.5$ & 50 & 5.29 & $1.78\cdot 10^{-4}$ & 0.028s \\
    $0.0001$ & $1.5$ & 20 & 5.29 & $3.43\cdot 10^{-9}$ & 0.488s \\
    $0.0001$ & $1.5$ & 50 & 5.29 & $3.43\cdot 10^{-9}$ & 2.463s \\
  \end{tabular}
\end{center}
\caption{Numerical test for Example~\ref{CDSysLinQuad5D} (5-dimensional deterministic).}\label{TabCDSysLinQuad5D}
\end{table}

\end{example}

From Table~\ref{TabCDSysLinQuad5D} and Figure~\ref{FigDetConvergence}, we see that convergence at initial state is fast during the first iterations. Indeed, twenty iterations are enough to reach a very small error, after that there is hardly any improvement.
  
Computation time is expected to scale (sub-)linearly in the number of time steps, while it will scale super-linearly in the number of iterations (because of the way subsolutions are represented). This can be observed in the data. Notice that with $h = 0.01$ there are $200$ steps in time, while a choice of $h = 0.0001$ corresponds to $20000$ time steps at each iteration.
  
For $c \in \{0.0, 0.5\}$, the error (both absolute and relative) in Table~\ref{TabCDSysLinQuad5D} is close to machine precision. In theory, the difference between value from above and value from below should be non-negative. The minus sign is likely caused by rounding errors, which however do not build up. For $c = 1.5$, we still have a small error, with relative error at $0.0034\%$ when $h = 0.01$, by orders of magnitude lower when $h = 0.0001$. The reason for this improvement when refining the time discretization is due to the ``Hamiltonian approximation'' we use in computing control actions (static optimization). The difference between value from above and value from below is non-negative, as expected from theory.

Notice that our algorithm does not take advantage of the affine-quadratic form of the terminal costs in Example~\ref{CDSysLinQuad5D}. Convergence for $L^{1}$-costs, for instance, would be even faster, as the corresponding value functions are more easily approximated through maxima of hyperplanes.

\begin{figure}[t] 
  \noindent\makebox[\textwidth]{
    \includestandalone[width=\textwidth, mode=image|tex]{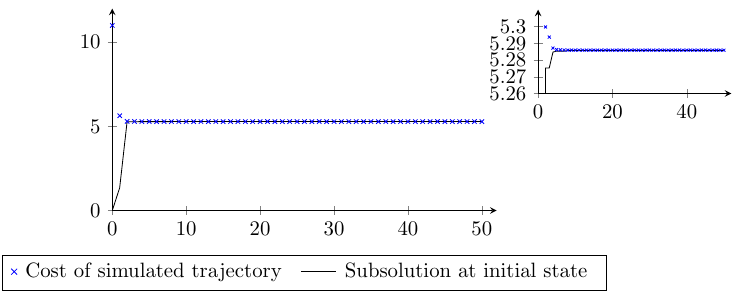}
  }
  \caption{Convergence at initial state for Example~\ref{CDSysLinQuad5D} (5-dimensional deterministic) with $h = 0.01$, $c = 1.5$.}\label{FigDetConvergence}
\end{figure}

\new{
To get a more precise idea on how computation time grows with the dimension of state and control space, we consider a simple $N$-particle system with mean field interaction. It models the friction-less motion in space of $N$ particles with acceleration subject to control and to a force coming from the average position of all particles. Running costs are quadratic in the control, while the quadratic terminal costs are minimal when the positions of the first half of the particles (more precisely, $\lfloor N/2 \rfloor$) are at $(1,1,1)$, and those of the second half at $(-1,-1,-1)$. Initial positions are chosen randomly and independently according to the uniform distribution on the three-dimensional sphere with radius $5$, while initial velocities are equal to zero.  

\begin{example}[$N$-particle deterministic] \label{CDSysNParticleMF}
  Dimensions $d = 6N$, $d_2 = 3N$, where $N\in \mathbb{N}$ number of particles, control actions in the $d_2$-dimensional closed ball $\Gamma \doteq \bar{B}(0,2)$, 
  \begin{align*}
  &A = (a_{kl}) \text{ with } a_{kl} = \begin{cases}
  1 &\text{if } (k\! \mod 6) \in \{1,2,3\} \text{ and } l=k+3, \\
  -\frac{1}{4N} &\text{if } (k\! \mod 6) \in \{4,5,0\} \text{ and } (|l-k|\! \mod 6) = 3, \\
  0 &\text{else},
  \end{cases} \\
  & B = (b_{kl}) \text{ with } b_{kl} = \begin{cases}
  1 &\text{if } k = 6i+3+j \text{ and } l=3i+j \\
  &\text{for some } i\in \{0,\ldots,N-1\},\; j\in \{1,2,3\},\\
  0 &\text{else},
  \end{cases} \\
  &\bar{f} \equiv 0,\; c = \frac{1}{2N}, \quad F(x) \doteq 1 + \frac{1}{N} \sum_{j=1}^{3} \left( \sum_{i=0}^{\lfloor N/2 \rfloor - 1}  (x_{6i+j}-1)^{2} + \sum_{i=\lfloor N/2\rfloor}^{N-1} (x_{6i+j}+1)^{2} \right).
  \end{align*}
  Table~\ref{TabCDNParticleMF} summarizes data of a numerical test with time horizon $T = 2$ and randomly chosen initial configuration $x_{0}$ such that
  \begin{align*}
    & \sum_{j=1}^{3} x_{0,6i+j}^2 = 25,&  &\sum_{j=1}^{3} x_{0,6i+3+j}^2 = 0, & &i\in \{0,\ldots,N-1\}.&
  \end{align*}
  The time discretization parameter is set to $h = 0.0001$, corresponding to $20000$ time steps at each iteration. Computations were performed with $20$ iterations at each run as the number of particles $N$ goes through the set $\{2, 5, 10, 20, 40, 50, 80\}$. The first column of Table~\ref{TabCDNParticleMF} gives the particle number $N$, while the second and third column show the corresponding dimensions of the state space $\mathcal{X} = \mathbb{R}^{d}$ and the control space $\Gamma \doteq \bar{B}(0,2) \subset \mathbb{R}^{d_2}$. The fourth column reports the value of the subsolution at initial state computed after $20$ iterations. The difference between this value from below and the value from above, namely the cost associated with the trajectory obtained from the feedback strategy induced by the last subsolution, is given in column five. The second but last column shows the computing time in seconds, while the last column shows values of the function $z \mapsto \frac{0.963}{4}\cdot z^2$ evaluated in $N$. 
  
\end{example}

\begin{table} 
  \begin{center}
    \begin{tabular}{l|c|c|c|c|r|r}
      $N$ & $d$ & $d_2$ & value below & $\Delta_{(above - below)}$ & CPU time & $\frac{0.963}{4}\cdot N^2$\\
      \hline
      2 & 12 & 6 & 7.58 & $5.87\cdot 10^{-10}$ & 0.963s & 0.963 \\
      5 & 30 & 15 & 11.31 & $2.80\cdot 10^{-10}$ & 3.357s & 6.019 \\
      10 & 60 & 30 & 16.74 & $1.31\cdot 10^{-11}$ & 12.910s & 24.075 \\
      20 & 120 & 60 & 18.34 & $2.10\cdot 10^{-12}$ & 51.085s & 96.300 \\
      40 & 240 & 120 & 24.38 & $-4.12\cdot 10^{-11}$ & 220.427s & 385.200 \\
      50 & 300 & 150 & 24.11 & $-2.14\cdot 10^{-11}$ & 338.211s & 601.875 \\
      80 & 480 & 240 & 22.40 & $-2.59\cdot 10^{-11}$ & 897.943s & 1540.800 \\[1ex]
      80 & 480 & 240 & 22.40 & $-2.59\cdot 10^{-11}$ & 899.017s & 1540.800\\
    \end{tabular}
  \end{center}
  \caption{Numerical test for Example~\ref{CDSysNParticleMF} ($N$-particle deterministic).}\label{TabCDNParticleMF}
\end{table}

From Table~\ref{TabCDNParticleMF} we see that computation time grows sub-quadratically in the number of particles and, equivalently, in the state and control space dimensions. We have re-scaled running and terminal costs in accordance with the mean field interaction in the dynamics. Minimal costs are thus expected to converge to some optimal limit costs as $N$ tends to infinity, in a way analogous to the law of large numbers. Randomness enters the system through the choice of the initial configuration. The difference between the value of the subsolution at initial state computed after $20$ iterations (value from below) and the cost associated with the trajectory obtained from the feedback strategy induced by the last subsolution (value from above) approaches machine precision. The instances where it is negative ($N \in \{40, 50, 80\}$) are likely due to rounding errors. Those errors do not build up nor do they seem to grow with the number of particles.

The last row in Table~\ref{TabCDNParticleMF} shows a second run with $N = 80$, starting from the same (randomly sampled) initial configuration as for the one reported in the first row with $N = 80$. While for that run, and the others reported above, the initial configuration stays the same for all 20 iterations, in this run every other iteration uses a perturbed initial configuration, which is obtained by adding to each component of the original configuration (position as well as velocity entries) the outcome of sampling from a sequence of i.i.d.\ Gaussian random variables with mean zero and variance equal to $0.25$. Figure~\ref{FigDetNParticlePerturbed} shows the value of the current subsolution and the cost of the simulated trajectory under the Markov feedback strategy induced by that subsolution as a function of the iteration step. We see quick convergence for even step numbers, namely those where the original initial configuration is used. For odd step numbers, we do not have, nor do we expect to have, fast convergence of values, as those iterations are computed with respect to ever new initial configurations, obtained by independent perturbations of the original one. Better performance of those feedback controls could be obtained by letting more than one iteration start from any of the perturbed configurations. Still, Figure~\ref{FigDetNParticlePerturbed} can be seen as an illustration of the claim that the subsolutions computed by our algorithm yield useful sub-optimal feedback controls also for previously unseen initial configurations.

\begin{figure}[t] 
\noindent\makebox[\textwidth]{
  \includestandalone[width=\textwidth, mode=image|tex]{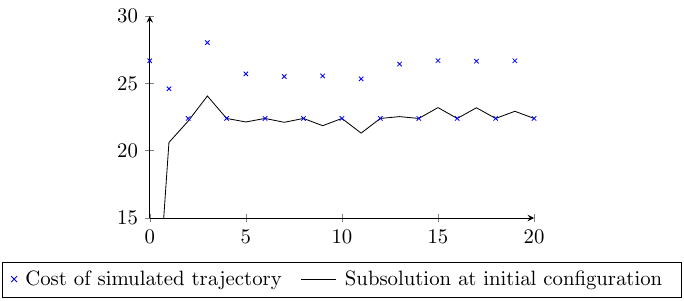}
  }
  \caption{Values at initial configuration for Example~\ref{CDSysNParticleMF} ($N$-particle deterministic) with $N = 80$; every other iteration (odd numbers) started from a perturbed initial configuration.}\label{FigDetNParticlePerturbed}
\end{figure}

}

\subsection{Stochastic case}

Given an initial state $X_{0} = x_{0}\in \mathbb{R}^{d}$, the optimal control problem is to minimize
\[
  \Mean\left[ \sum_{j=0}^{N-1} \left( c|\gamma_{j}|^{2} + \bar{f}(X_{j}) \right)h + F(X_{N})\right],\quad \text{where }N \doteq T/h,
\]
over $\bar{B}(0,r)$-valued adapted control processes $(\gamma_{j})$ subject to the dynamics
\[
  X_{j+1} = X_{j} + \left(A X_{j} + B \gamma_{j}\right)h + \sqrt{h}C\xi_{j+1}, \quad j\in \{0,\ldots,N-1\},
\]
where $\xi_{1}, \xi_{2}, \ldots$ are i.i.d.\ $\mathbb{R}^{d_1}$-valued random variables with common distribution $\mu \doteq \otimes^{d_{1}} \mathrm{Rad}(1/2)$.

\begin{example}[5-dimensional stochastic] \label{CSSysLinQuad5D}
Dimensions $d = d_{1} = d_2 = 5$, control actions in the closed ball $\Gamma \doteq \bar{B}(0,1)$,
\begin{align*}
  &A = 0,& &B = \Id,&  &C = 0.25\cdot \Id,& &\bar{f} \equiv 0,& & F(x) \doteq 1 + \|x\|^{2}.&
\end{align*}
Table~\ref{TabCSSysLinQuad5D} summarizes data of a numerical test with discretization parameter $h = 0.01$, time horizon $T = 2$, and initial state $x_{0}\doteq (1, -\sqrt{3}, 2, 1, -1)$ starting from the constant subsolution $w^{(0)} \equiv 0$. The first column there gives the cost coefficient $c$ for the quadratic control running costs. The second column shows the number of iterations of the algorithm, while the third column reports the value of the subsolution at initial state computed after the indicated number of iterations. The difference between this value from below and an estimated value from above, which is computed through Monte Carlo simulation as the average cost associated with $10^{4}$ trajectories generated according to the feedback strategy induced by the last subsolution, is given in column four. Column five shows the corresponding relative error: (estimated value from above minus value from below) divided by value from below. The computing time in seconds appears in column six. For our simple example, an optimal feedback strategy can be calculated analytically. The difference between a cost estimate corresponding to that optimal feedback strategy (again computed by Monte Carlo simulation) and the estimated value from above is given in the last column.

\begin{table} 
\begin{center}
  \begin{tabular}{c|c|c|c|c|c|c}
    $c$ & \#iter. & value below & $\Delta_{(above - below)}$ & relative $\Delta$ & CPU & $\Delta_{(above - opt)}$ \\
    \hline
    $0.0$ & 20 & 2.15 & 0.751 & 34.93\% & 0.126s & 0.1220 \\
    $0.0$ & 50 & 2.27 & 0.593 & 26.12\% & 0.434s & 0.0783 \\
    $0.0$ & 100 & 2.28 & 0.562 & 24.65\% & 1.440s & 0.0600 \\
    $0.0$ & 200 & 2.29 & 0.533 & 23.28\% & 5.446s & 0.0385 \\[1ex]
    $0.5$ & 20 & 3.15 & 0.751 & 23.84\% & 0.124s & 0.1189 \\
    $0.5$ & 50 & 3.27 & 0.592 & 18.10\% & 0.429s & 0.0784 \\
    $0.5$ & 100 & 3.28 & 0.561 & 17.10\% &1.414s & 0.0589 \\
    $0.5$ & 200 & 3.29 & 0.532 & 16.17\% & 5.460s & 0.0380 \\[1ex]
    $1.5$ & 20 & 5.13 & 0.697 & 13.59\% & 0.121s & 0.1177 \\
    $1.5$ & 50 & 5.13 & 0.653 & 12.73\% & 0.429s & 0.0778 \\
    $1.5$ & 100 & 5.14 & 0.620 & 12.06\% & 1.404s & 0.0555 \\
    $1.5$ & 200 & 5.15 & 0.590 & 11.46\% & 5.338s & 0.0349 \\
  \end{tabular}
\end{center}
\caption{Numerical test for Example~\ref{CSSysLinQuad5D} (5-dimensional stochastic).}\label{TabCSSysLinQuad5D}
\end{table}
\end{example}

\begin{figure}[t] \label{FigStConvergence}
    \noindent\makebox[\textwidth]{
      \includestandalone[width=\textwidth, mode=image|tex]{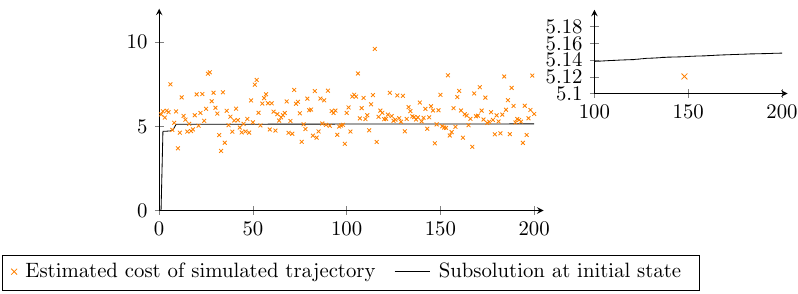}
    }
  \caption{Convergence at initial state for Example~\ref{CSSysLinQuad5D} (5-dimensional stochastic) with $h = 0.01$, $c = 1.5$.}
\end{figure}

From Table~\ref{TabCSSysLinQuad5D} and Figure~\ref{FigStConvergence}, we see that convergence is much slower than in the deterministic case. After the first 100 iterations, improvements in the approximation error become small. After 200 iterations, we still have a relative error between 24\% ($c = 0.0$) and 11\% ($c = 1.5$). As in the deterministic case, convergence speed may strongly depend on the choice of parameter values like the cost coefficient $c$. \new{In this example, if $c = 0.0$, then we have zero running costs (since $\bar{f} \equiv 0$ here). For our subsolutions to come close to the value function at initial time, all costs have to be propagated backwards from terminal time, as the terminal costs are the only source of costs when $c = 0.0$. For $c > 0$, running costs contribute to the value at initial time along the entire state trajectory, and they are propagated backwards more quickly. This may explain why the relative error (for fixed number of iterations) decreases as the cost coefficient $c$ increases.}

The computational cost is higher than in the deterministic case since evaluation of the Bellman operator now requires to compute expected values. As we have chosen a random walk discretization of the noise, expected values in dimension five correspond to convex combinations of $2^5 = 32$ summands. This leads to a ``built-in'' curse-of-dimensionality with respect to $d_1$, the noise dimension. The overall computation time, as compared to Example~\ref{CDSysLinQuad5D}, increases nevertheless only by a factor around five. The main computational burden actually lies in the Monte Carlo simulation for the cost estimates when using a subsolution produced after several hundred iterations. 

For Example~\ref{CSSysLinQuad5D}, an optimal feedback strategy can be found explicitly. Using that strategy to compare costs, we see that the feedback strategies induced by the subsolutions generated by our algorithm do not perform too badly; cf.\ last column of Table~\ref{TabCSSysLinQuad5D}.  In fact, the relative error after 200 iterations lies between 1.7\% ($c = 0.0$) and 0.6\% ($c = 1.5$), an order of magnitude better than the relative error in terms of the value at initial state.

\begin{example}[Brownian particle in space] \label{CSSysOUMotionSpace}
Dimensions $d = 6$, $d_{1} = d_2 = 3$, control actions in the closed ball $\Gamma \doteq \bar{B}(0,2)$,
\begin{align*}
  &A = \begin{pmatrix}
   0.0 & 0.0 & 0.0 & 1.0 & 0.0 & 0.0 \\ 
   0.0 & 0.0 & 0.0 & 0.0 & 1.0 & 0.0 \\ 
   0.0 & 0.0 & 0.0 & 0.0 & 0.0 & 1.0 \\ 
   0.0 & 0.0 & 0.0 & -0.2 & 0.0 & 0.0 \\ 
   0.0 & 0.0 & 0.0 & 0.0 & -0.2 & 0.0 \\ 
   0.0 & 0.0 & 0.0 & 0.0 & 0.0 & -0.2
  \end{pmatrix},&
  &B = \begin{pmatrix}
   0.0 & 0.0 & 0.0 \\ 
   0.0 & 0.0 & 0.0 \\ 
   0.0 & 0.0 & 0.0 \\ 
   1.0 & 0.0 & 0.0 \\ 
   0.0 & 1.0 & 0.0 \\ 
   0.0 & 0.0 & 1.0 
  \end{pmatrix},& \\
  &C = \begin{pmatrix}
     0.0 & 0.0 & 0.0 \\ 
     0.0 & 0.0 & 0.0 \\ 
     0.0 & 0.0 & 0.0 \\ 
     0.25 & 0.0 & 0.0 \\ 
     0.0 & 0.25 & 0.0 \\ 
     0.0 & 0.0 & 0.25 
    \end{pmatrix},&
  &\bar{f}(x)\doteq \frac{x_{1}^{2} + x_{2}^{2} + x_{3}^{2}}{2},\quad F\equiv 1.&
\end{align*}
Table~\ref{TabCSSysOUMotionSpace} summarizes data of a numerical test with running cost coefficient $c = 0.5$, time horizon $T = 2$, and initial state $x_{0}\doteq (1, -3, 2, 0, 0, 0)$ starting from the constant subsolution $w^{(0)} \equiv 0$.

\begin{table} 
\begin{center}
  \begin{tabular}{l|c|c|c|c|r}
    $h$ & \#iter. & value below & $\Delta_{(above - below)}$ & relative $\Delta$ & CPU time\\
    \hline
    $0.01$ & 20 & 10.78 & 0.1087 & 1.01\% & 0.049s \\
    $0.01$ & 50 & 10.78 & 0.0960 & 0.89\% & 0.159s \\
    $0.01$ & 100 & 10.79 & 0.0914 & 0.85\%  & 0.479s \\
    $0.01$ & 200 & 10.79 & 0.0855 & 0.79\%  & 1.675s \\[1ex]
    $0.001$ & 20 & 10.73 & 0.0937 & 0.87\% & 0.416s \\
    $0.001$ & 50 & 10.73 & 0.0907 & 0.84\% & 1.509s \\
    $0.001$ & 100 & 10.74 & 0.0851 & 0.79\%  & 4.700s \\
    $0.001$ & 200 & 10.74 & 0.0824 & 0.77\%  & 16.641s \\
  \end{tabular}
\end{center}
\caption{Numerical test for Example~\ref{CSSysOUMotionSpace} (\new{Brownian particle in space}).}\label{TabCSSysOUMotionSpace}
\end{table} 
\end{example}

From Table~\ref{TabCSSysOUMotionSpace} we find, as in the previous examples, fast initial convergence. A relative error around $1\%$ is reached after the first twenty iterations. The error is slightly lower for discretization parameter $h = 0.001$ than for $h = 0.01$, again because of the Hamiltonian approximation in computing one-step control actions. As expected, computation time increases by a factor of around ten when passing from $200$ time steps ($h = 0.01$) to $2000$ ($h = 0.001$). As in the deterministic case, we see that convergence speed strongly depends on the problem structure. Indeed, convergence is much better for Example~\ref{CSSysOUMotionSpace} than for Example~\ref{CSSysLinQuad5D}, thanks to the fact that only three of the six state components are controlled and affected by noise.

\new{
Lastly, we create a stochastic version of the $N$-particle system with mean field interaction from Example~\ref{CDSysNParticleMF} by adding a common noise to the acceleration components of all particles. The noise space will thus always be three-dimensional, while state and control space dimensions are proportional to the number of particles $N$. Running costs as well as terminal costs are the same as for the deterministic version, and also initial positions and velocities are chosen in the same way as there.
  
\begin{example}[$N$-particle stochastic with common noise] \label{CSSysNParticleMF}
  Dimensions $d = 6N$, $d_1 = 3$, $d_2 = 3N$, where $N\in \mathbb{N}$ number of particles, control actions in the $d_2$-dimensional closed ball $\Gamma \doteq \bar{B}(0,2)$, 
    \begin{align*}
    &A = (a_{kl}) \text{ with } a_{kl} = \begin{cases}
    1 &\text{if } (k\! \mod 6) \in \{1,2,3\} \text{ and } l=k+3, \\
    -\frac{1}{4N} &\text{if } (k\! \mod 6) \in \{4,5,0\} \text{ and } (|l-k|\! \mod 6) = 3, \\
    0 &\text{else},
    \end{cases} \\
    & B = (b_{kl}) \text{ with } b_{kl} = \begin{cases}
    1 &\text{if } k = 6i+3+j \text{ and } l=3i+j \\
    &\text{for some } i\in \{0,\ldots,N-1\},\; j\in \{1,2,3\},\\
    0 &\text{else},
    \end{cases} \\
    & C = (c_{kl}) \text{ with } c_{kl} = \begin{cases}
    0.25 &\text{if } k = 6i+3+j \text{ and } l=j \\
    &\text{for some } i\in \{0,\ldots,N-1\},\; j\in \{1,2,3\},\\
    0 &\text{else},
    \end{cases} \\
    &\bar{f} \equiv 0,\; c = \frac{1}{2N}, \quad F(x) \doteq 1 + \frac{1}{N} \sum_{j=1}^{3} \left( \sum_{i=0}^{\lfloor N/2 \rfloor - 1}  (x_{6i+j}-1)^{2} + \sum_{i=\lfloor N/2\rfloor}^{N-1} (x_{6i+j}+1)^{2} \right).
    \end{align*}
    Table~\ref{TabCSNParticleMF} summarizes data of a numerical test with time horizon $T = 2$ and randomly chosen initial configuration $x_{0}$ such that
    \begin{align*}
    & \sum_{j=1}^{3} x_{0,6i+j}^2 = 25,&  &\sum_{j=1}^{3} x_{0,6i+3+j}^2 = 0, & &i\in \{0,\ldots,N-1\}.&
    \end{align*}
  The time discretization parameter is set to $h = 0.001$, corresponding to $2000$ time steps at each iteration. The number of particles $N$ goes through the set $\{1, 2, 5, 10\}$. The first column of Table~\ref{TabCSNParticleMF} gives the particle number $N$, while the second and third column show the corresponding dimensions of the state space $\mathcal{X} = \mathbb{R}^{d}$ and the control space $\Gamma \doteq \bar{B}(0,2) \subset \mathbb{R}^{d_2}$. The common noise is always three-dimensional. 
    
  The fourth column shows the number of iterations of the algorithm, equal to either $20$ or $50$, while the fifth column reports the value of the subsolution at initial state computed after the indicated number of iterations. The difference between this value from below and an estimated value from above, which is computed through Monte Carlo simulation as the average cost associated with $10^{4}$ trajectories generated according to the feedback strategy induced by the last subsolution, is given in column six. Column seven shows the corresponding relative error: (estimated value from above minus value from below) divided by value from below. The computing time in seconds appears in the last column. 
    
  \end{example}

  \begin{table} 
    \new{
      \begin{center}
        \begin{tabular}{l|c|c|c|c|c|c|r}
          $N$ & $d$ & $d_2$ & \#iter. & value below & $\Delta_{(above - below)}$ & relative $\Delta$ & CPU time \\
          \hline
          1 & 6 & 3 & 20 & 2.42 & 0.283 & 11.69\% & 0.290s \\
          1 & 6 & 3 & 50 & 2.42 & 0.237 & 9.82\% & 1.186s \\
          2 & 12 & 6 & 20 & 7.48 & 0.421 & 5.63\% & 0.596s \\
          2 & 12 & 6 & 50 & 7.49 & 0.404 & 5.39\% & 2.093s \\
          5 & 30 & 15 & 20 & 11.26 & 0.408 & 3.63\% & 2.054s \\
          5 & 30 & 15 & 50 & 11.26 & 0.402 & 3.57\% & 6.438s \\
          10 & 60 & 30 & 20 & 16.58 & 0.536 & 3.23\% & 7.230s \\
          10 & 60 & 30 & 50 & 16.66 & 0.451 & 2.71\% & 20.814s\\
        \end{tabular}
      \end{center}
      \caption{Numerical test for Example~\ref{CSSysNParticleMF} ($N$-particle stochastic with three-dimensional common noise).}\label{TabCSNParticleMF}
    }
  \end{table}
  
  Table~\ref{TabCSNParticleMF} suggests that computation time grows sub-quadratically in the number of particles and, equivalently, in the state and control space dimensions. We have re-scaled running and terminal costs in accordance with the mean field interaction in the dynamics. Minimal costs are thus expected to converge to some optimal limit costs as $N$ tends to infinity, in a way analogous to the law of large numbers (with a common noise). This may explain why the relative error actually decreases as the number of particles increases. 

}

\bigskip
In conclusion, regarding our numerical tests, for both deterministic and stochastic problems we find good initial convergence. The actual speed of convergence strongly depends on the problem structure as well as parameter values. Thanks to the subsolution property, one can always obtain a posteriori error bounds. For deterministic problems, we find fast convergence with errors getting down to machine precision \new{even for problems with several hundred state and control dimensions}. Computation time scales \new{(sub-)linearly in the number of time steps, while it appears to scale sub-quadratically in the number of iterations as well as the number of state and control dimensions.} Convergence in terms of the value at initial state can be much slower for stochastic problems. The subsolution generated by our algorithm may nevertheless yield a sub-optimal feedback strategy that is useful in terms of associated costs.


\bibliographystyle{plain}

\end{document}